\documentclass[10pt,a4paper,twoside,reqno]{amsart}
\usepackage[a4paper,top=3cm,bottom=3cm,inner=2.5cm,outer=2.5cm]{geometry}

\usepackage[utf8]{inputenc}
\usepackage[english]{babel}

\usepackage{xcolor}

\usepackage{amssymb,mathtools,amsmath,amsthm,stmaryrd}
\usepackage[all,cmtip,2cell]{xy} 

\usepackage[draft=false, hidelinks, linkcolor = blue]{hyperref}
\usepackage{cleveref}
\usepackage{tensor}
\usepackage{bbm}
\usepackage{url}
\usepackage{bookmark}
\usepackage{footmisc}

\usepackage{graphicx}
\usepackage{lmodern}
\usepackage{enumitem}
\usepackage{array}

\def\defthm#1#2#3#4{
  \newtheorem{#1}[theorem]{#3}
  \newtheorem*{#1*}{#3}
  \newtheorem{#2}[theorem]{#4}
  \newtheorem*{#2*}{#4}
  \crefname{#1}{#3}{#4}
  \crefname{#2}{#4}{#4}  
}

\newtheoremstyle{mythm}%
{10pt}
{}
{\itshape}
{}
{\bf}
{.}
{.5em}
{}%

\newtheoremstyle{mydef}%
{10pt}
{3pt}
{}
{}
{\bf}
{.}
{.5em}
{}%

\newtheoremstyle{myrmk}%
{10pt}
{3pt}
{}
{}
{\bf}
{.}
{.5em}
{}%

\theoremstyle{mythm}
\newtheorem{theorem}{Theorem}[section]
\newtheorem*{theorem*}{Theorem}

\defthm{corollary}{corollaries}{Corollary}{Corollaries}
\defthm{lemma}{lemmata}{Lemma}{Lemmata}
\defthm{proposition}{propositions}{Proposition}{Propositions}
\defthm{axiom}{axioms}{Axiom}{Axioms}
\defthm{propdef}{props/defs}{Proposition/Definition}{Prositions/Definitions}
\defthm{defcor}{defscors}{Corollary/Definition}{Corollaries/Definitions}
\defthm{conjecture}{conjectures}{Conjecture}{Conjectures}

\theoremstyle{mydef}
\defthm{definition}{definitions}{Definition}{Definitions}

\theoremstyle{myrmk}
\defthm{acknowledgment}{acknowledgments}{Acknowledgment}{Acknowledgments}
\defthm{remark}{remarks}{Remark}{Remarks}
\defthm{motivation}{motivations}{Motivation}{Motivations}
\defthm{example}{examples}{Example}{Examples}
\defthm{question}{questions}{Question}{Questions}
\defthm{observation}{observations}{Observation}{Observations}
\defthm{claim}{claims}{Claim}{Claims}
\defthm{notation}{notations}{Notation}{Notations}
\defthm{terminology}{terminologies}{Terminology}{Terminologies}
\defthm{construction}{constructions}{Construction}{Constructions}
\defthm{conclusion}{conclusions}{Conclusion}{Conclusions}

\newtheorem*{replemmax}{\reptitle}
 {\end{replemmax}}

\newtheorem*{repthmx}{\reptitle}
\newenvironment{repthm}[1]{%
 \def\reptitle{Theorem \ref*{#1}}%
 \begin{repthmx}}%
 {\end{repthmx}}

\newtheorem*{repcorx}{\reptitle}
 {\end{repcorx}}

\newcommand{\twocat}{2\text{-}\mathrm{Cat}}

\begin{document}

\title[Symmetry shifting for monoidal bicategories]{
Symmetry shifting for monoidal bicategories%
}
\author[R. Stenzel] {Raffael Stenzel}

\begin{abstract}
We show that every braiding on a monoidal bicategory induces a monoidal structure on its bicategory of monoids, such that 
if the former is sylleptic or symmetric then the latter is braided or symmetric, respectively. This extends a classic theorem of 
Joyal and Street for monoidal categories. The proof presented in this paper is an application of the $\infty$-operadic Additivity Theorem and 
thereby averts any considerable calculations.
\end{abstract}

\date{\today}

\maketitle

\section{Introduction}\label{sec_intro}

Categorified higher dimensional algebra is the focal point of a variety of mathematical activity in both foundations and applications 
across algebra, topology, geometry, physics and logic. This applies particularly to the fields of topological quantum field theories,
K-theory and representation theory \cite{baezdolan_cat, schommerpries_thesis, lurieha}. 
Indeed, monoidal category theory was part of Mac Lane's standard canon of category theory already \cite{maclane}, and since has 
flourished into a well-developed theory with applications in a plethora of disciplines \cite{joyalstreet}. The study of 
monoidal bicategory theory was initiated by Kapranov and Voevodsky \cite{kapranovvoevodskyI,kapranovvoevodskyII}, and since has 
undergone substantial development as well \cite{johnsonyau}. The 2-dimensional theory is 
motivated both by way of its applications in physics and topology (to remedy short-comings of its 1-dimensional counterpart) as well as 
by its fundamental role in the formalisation of low dimensional category theory in general \cite{gurski_3coh}. The fully 
algebraic theory in two dimensions however is notoriously difficult to navigate due to the sheer number and size of coherence axioms to 
consider at all times. This applies all the more to theories of categorified algebra in dimension 3 \cite{hoffnung_tetra}, or higher.
On the other side of the spectrum, monoidal $\infty$-category theory has more recently been developed following the operadic tradition of 
May instead \cite{may_loop, lurieha}. Here, the coherence axioms are built into the machinery by homotopy theoretical means, which further 
opens the opportunity to use topological and geometric ideas as well.

This paper constitutes the second of a series on the $\infty$-categorical algebra of bicategories. 
Its ultimate aim is to use the results of higher algebra as pivotally developed in \cite{lurieha} to give fairly 
straightforward proofs of otherwise computationally long-winded results in low dimensional category theory. 
To motivate the main result of this paper as well as the convenience of higher algebra to prove it, we recall the following well-known result 
about monoidal categories.

\begin{repthm}{cor_main1cat}[Joyal--Street]
Let $\mathcal{C}$ be a monoidal category and $\mathrm{Mon}(\mathcal{C})$ the category of monoids in $\mathcal{C}$.
\begin{enumerate}
\item Every braiding on $\mathcal{C}$ induces a monoidal structure on $\mathrm{Mon}(\mathcal{C})$.
\item Every symmetry on $\mathcal{C}$ induces a symmetric monoidal structure on $\mathrm{Mon}(\mathcal{C})$.
\end{enumerate}
In both cases the corresponding structure is canonically induced, i.e.\ the forgetful functor
$\mathrm{Mon}(\mathcal{C})\rightarrow\mathcal{C}$ is as monoidal as its domain is.
\end{repthm}

We refer to Theorem~\ref{cor_main1cat} as symmetry shifting for monoidal categories, the reason of which will become clear below.
To prove the theorem, one may without loss of generality assume that the monoidal category $\mathcal{C}$ is a strict monoidal category by 
way of Mac Lane's celebrated Coherence Theorem. Then assuming that $\mathcal{C}$ is equipped with a braiding $\beta$, the tensor product of 
two monoids $A$ and $B$ can be explicitly computed by way of the tensor product of their underlying objects in $\mathcal{C}$. Its multiplication is given by the composition
\begin{align}\label{equ_intro_mono_mult}
\scalebox{0.9}{$
(A\otimes B)\otimes (A\otimes B)\xrightarrow{=}A \otimes (B\otimes A)\otimes B\xrightarrow{\beta}A\otimes (A\otimes B)\otimes B\xrightarrow{=}(A\otimes A)\otimes (B\otimes B)\xrightarrow{m_A\otimes m_B}A\otimes B.$}
\end{align}

The due axioms are manageable to be computed by hand. The main result of this paper is the following symmetry shifting theorem in dimension 
2.

\begin{repthm}{mainthmgray}
Let $\mathcal{C}$ be a monoidal bicategory and $\mathrm{Mon}(\mathcal{C})$ be the bicategory of monoids in $\mathcal{C}$.
\begin{enumerate}
\item Every braiding on $\mathcal{C}$ induces a monoidal structure on $\mathrm{Mon}(\mathcal{C})$.
\item Every syllepsis on $\mathcal{C}$ induces a braided monoidal structure on $\mathrm{Mon}(\mathcal{C})$.
\item Every symmetry on $\mathcal{C}$ induces a symmetric monoidal structure on $\mathrm{Mon}(\mathcal{C})$.
\end{enumerate}
In each case the corresponding structure is canonically induced, i.e.\ the forgetful functor
$\mathrm{Mon}(\mathcal{C})\rightarrow\mathcal{C}$ is as monoidal as its domain is.
\end{repthm}

Once again, to prove the theorem, one may without loss of generality assume that the monoidal bicategory $\mathcal{C}$ is a Gray monoid by 
way of Gordon, Power and Street's Coherence Theorem \cite{gps_coherence}. Here however, first, one is to be careful with the precise 
definition of monoidal structures, braidings, syllepses and symmetries. Second, even after the process of semi-strictifying $\mathcal{C}$ to 
a Gray monoid, the direct computations are much longer and involved simply by virtue of the amount of axioms to verify individually.
And third, there clearly is a pattern emerging when considering Theorem~\ref{cor_main1cat} and Theorem~\ref{mainthmgray} in sequence.

The primary goal of this paper is to formulate and prove the following $\infty$-categorical version of Theorem~\ref{mainthmgray} and thereby 
address all three of the above comments.

\begin{repthm}{thm_main_12}
Let $\mathcal{C}$ be an $\mathsf{E}_k$-monoid in the cartesian monoidal $\infty$-category $\text{BiCat}^{\times}$ of 
bicategories. Let $0\leq m\leq k$ be an integer. Then the bicategory $\mathrm{Alg}_{\mathsf{E}_m}(\mathcal{C})$ of $m$-fold pseudomonoids 
in $\mathcal{C}$ carries a canonical $E_{k-m}$-monoid structure.\footnote{Here, for $k=\infty$, $0\leq m\leq \infty$, we set
$k - m =\infty$.} Furthermore, the bifunctor \[U\colon\mathrm{Alg}_{\mathsf{E}_m}(\mathcal{C})\rightarrow \mathcal{C}\]
which assigns to an $m$-fold pseudomonoid its underlying object lifts to a morphism of $\mathsf{E}_{k-m}$-monoids.
\end{repthm}

\begin{notation*}
Although potentially confusing in the context of the notions at hand, the term ``$\infty$-category'' refers to ``$(\infty,1)$-category'' 
following the common conventions. And since we use the results of \cite{lurieha} extensively throughout this paper, this more 
precisely shall mean ``quasi-category'' specifically.
\end{notation*}

Informally, Theorem~\ref{thm_main_12} states that the assignment $\mathrm{Alg}_{\mathsf{E}_m}(-)$ introduces a shift on the degree of 
symmetry of its input by $m$. This motivates the name of the two theorems above, and hence of the paper. To understand how 
Theorem~\ref{thm_main_12} and Theorem~\ref{mainthmgray} are related, we note that, up to equivalence, 
the $\mathsf{E}_1$-monoids in $\text{BiCat}^{\times}$ stand in canonical 1-1 correspondence to the monoidal bicategories in the original algebraic sense of Kapranov--Voevodsky. Similarly, the bicategory of $\mathsf{E}_1$-algebras in a monoidal bicategory $\mathcal{C}$ consists 
precisely of the monoids in $\mathcal{C}$ in the original algebraic sense.
In general, $\mathsf{E}_k$-algebras for larger integers $k$ are understood as associative algebra structures with a degree of 
commutativity, where full commutativity is recovered at $k=\infty$. The bound $0\leq m\leq k$ in Theorem~\ref{thm_main_12} is hence an 
instant of what is often referred to as the microcosm principle. The coining of this principle is commonly attributed to Baez and Dolan. It 
states that ``certain algebraic structures can be defined in any category equipped with a categorified version of the same structure'', or 
in other words, that ``every feature of the microcosm [\dots] corresponds to some feature of the macrocosm'' \cite{baezdolan3}.
In terms of this vernacular, the argument that lies at the heart of Theorem~\ref{thm_main_12}, and of Theorem~\ref{mainthmgray} subsequently,  
is the simple observation that the collection of macrocosms itself is monoidal, and that the assignment which maps a macrocosm $\mathcal{C}$ 
to the collection of microcosms in $\mathcal{C}$ is a monoidal functor. The statements then essentially follow formally by way of the 
Additivity Theorem. This will be explained in more detail in the Outline.

To relate the higher $\mathsf{E}_k$-structures on a bicategory to their algebraic counterparts, one is to identify the higher algebra of 
the cartesian monoidal $\infty$-category $\text{BiCat}^{\times}$ of bicategories in the sense of \cite{lurieha} with the
2-dimensional algebra of bicategories in the sense of \cite{daystreet} and \cite{joyalstreet} to an extent specified in 
Theorem~\ref{thm_translatealghom}. This identification, which is topic of the first paper of this series \cite{rs_halgicatI}, will allow to 
prove Theorem~\ref{mainthmgray} from Theorem~\ref{thm_main_12} with minimal indulgence in the pasting of daring diagrams which is typical for 
2-dimensional algebra otherwise. 

Subject to a further characterization of $m$-fold pseudomonoid structures inside a bicategory in terms of corresponding braidings 
that we leave open for future work, one can show a variation of Theorem~\ref{mainthmgray} that applies to 
the bicategory of pseudomonoids with extra monoidal structure as well (Theorem~\ref{mainthmgray2}).

\paragraph{Related Work}
Theorem~\ref{mainthmgray} was first discovered by Nick Gurski \cite{gurski_stability} and independently by Nicola Gambino, 
both of whom proved it by way of direct algebraic calculations. However, to the best of our knowledge, the literature does not 
contain a proof. A similar type of statement about the 2-category of modules associated to a braided (symmetric) monoid inside a
braided (sylleptic/symmetric) multifusion 2-category can be found in \cite[Theorems A,B,C]{decoppetyu2cat}. This also was proven by way of 
direct algebraic calculations.

\paragraph{Outline of the technical aspects}

To understand how the homotopy-coherent tools of higher algebra are useful to prove Theorem~\ref{mainthmgray}, we note that 
the theorem is a priori a statement about constructions in the tetracategory of reduced tricategories. That is, effectively, in the 
tricategory of monoidal bicategories. However, even though bicategories are not $(2,1)$-categories themselves, one quickly 
realizes that the structures in Theorem~\ref{mainthmgray} are in fact constructions in the $(3,1)$-category $\mathrm{BiCat}$ of 
bicategories. That is, because each occurring notion is defined in terms of higher invertible cells only.
Furthermore, one may observe that the bicategories $\mathrm{Mon}(\mathcal{C})$ are the objectwise values of a $(3,1)$-functor
\begin{align}\label{equ_intro_mon}
\mathrm{Mon}(-)\colon\mathrm{MonBicat}\rightarrow\mathrm{BiCat}
\end{align}
defined on the $(3,1)$-category of monoidal bicategories. To define the functor (\ref{equ_intro_mon}), we will use 
that the $(3,1)$-category $\mathrm{BiCat}$ can be realized as the usual underlying homotopy $\infty$-category of the Lack model structure 
on 2-categories. This allows us to reduce the problem of 3-dimensional functoriality of (\ref{equ_intro_mon}) to the construction of a 
homotopically well behaved 1-dimensional presentation thereof. The statements in Theorem~\ref{mainthmgray} then are a formal 
consequence of the following facts: 
\begin{enumerate}
\item We may replace the $(3,1)$-category $\mathrm{MonBicat}$ by the $(3,1)$-category
$\mathrm{Alg}_{\mathsf{E}_1}(\mathrm{BiCat}^{\times})$ of homotopy-coherent monoids in $\mathrm{BiCat}^{\times}$ 
(Lemma~\ref{lemma_grmon=E1}).
\item The functor (\ref{equ_intro_mon}) is cartesian monoidal (Corollary~\ref{cor_inftypsmon}). It hence induces lifts of the form
\[\mathrm{Alg}_{\mathsf{E}_k}(\mathrm{Mon}(-))\colon\mathrm{Alg}_{\mathsf{E}_{k+1}}(\mathrm{BiCat}^{\times})\xrightarrow{\simeq}\mathrm{Alg}_{\mathsf{E}_{k}}(\mathrm{Alg}_{\mathsf{E}_{1}}(\mathrm{Bicat}^{\times}))\rightarrow\mathrm{Alg}_{\mathsf{E}_{k}}(\mathrm{BiCat}^{\times}).\]
\item Braided, sylleptic and symmetric monoidal bicategories are all $\mathsf{E}_k$-algebras for corresponding integers $k$ each (Theorem~\ref{thm_translatealghom}).
\end{enumerate}
These three facts together essentially assure that the structures we are interested in are all corepresentable in a suitable
$\infty$-category, and that $\mathrm{Mon}(-)$ is a morphism in that $\infty$-category.
Theorem~\ref{mainthmgray} is hence an instance of postcomposition and the Additivity Theorem (recalled in Theorem~\ref{thm_add}).

One of the benefits of this proof is that it is not specific to 2-dimensional algebra but applies to other contexts as well (e.g. to
$\infty$-categories as in \ Section~\ref{sec_infty}, or to ordinary categories as in Section~\ref{sec_1cat}). In particular, despite its 
global character, it recovers the explicit local formula (\ref{equ_intro_mono_mult}) for the multiplication of two monoids in any individual 
monoidal category (Proposition~\ref{prop_mult_explicit}).

\paragraph{Organisation of the paper}
A reminder of the relevant $\infty$-categorical notions, as well as a purely $\infty$-categorical version of Theorem~\ref{thm_main_12}
for motivation is given in Section~\ref{sec_infty}. Section~\ref{sec_1cat} discusses the direct application of this fragment of higher 
algebra to the theory monoidal categories, again mostly for motivation. This includes an $\infty$-categorical proof of 
Theorem~\ref{cor_main1cat}. Section~\ref{sec_modbicat} recalls the 2-dimensional algebra of bicategories in terms of its semi-strict 
Gray monoidal presentation, discusses its associated homotopy theory, and defines the functor (\ref{equ_intro_mon}). It further uses the 
results of Section~\ref{sec_infty} to give a proof of Theorem~\ref{thm_main_12} and subsequently of Theorem~\ref{mainthmgray}.
Section~\ref{sec_future} discusses possible future work.
\begin{acknowledgments*}
This paper was written in order to answer a question posed by Nicola Gambino; without his impetus and continuous support the paper would not 
have come to be. The idea to use higher algebra in this context arose from suggestions by Clark Barwick and James Cranch. 
The author would like to thank John Bourke for his insightful comments, Amar Hadzihasanovic for his kind invite to Tallinn to discuss the 
results, and Adrian Miranda as well as Mike Shulman for helpful discussions. This material is based upon work supported by the US Air 
Force Office for Scientific Research under award number FA9550-21-1-0007.
\end{acknowledgments*}

\section{Monoidal $\infty$-categories and monoids therein}\label{sec_infty}
The first subsection is a brief reminder of the relevant definitions, theorems and examples in \cite{lurieha}.
In the second subsection we discuss a straightforward $\infty$-categorical version of Theorem~\ref{thm_main_12}. 

\subsection*{Brief preliminaries}

\paragraph{$\infty$-Operads}

The notion of an $\infty$-operad generalises the notion of a (coloured) symmetric multicategory in essentially the same way as the notion 
of an $\infty$-category generalises that of a category. More precisely, let $\mathrm{Fin}_{\ast}$ denote the category of finite pointed 
sets and base-point preserving functions. We consider $\mathrm{Fin}_{\ast}$ as an $\infty$-category. We recall that an $\infty$-operad
$\mathcal{O}\rightarrow\mathrm{Fin}_{\ast}$ is defined to be a morphism of simplicial sets that satisfies the Segal condition, and that 
further exhibits cocartesian lifts to all inert maps in $\mathrm{Fin}_{\ast}$ \cite[Definition 2.1.1.10]{lurieha}. Every such morphism
$\mathcal{O}\rightarrow\mathrm{Fin}_{\ast}$ is in particular an isofibration, and so the domain $\mathcal{O}$ of an $\infty$-operad is 
always an $\infty$-category. The latter is sometimes referred to as the $\infty$-category of operators associated to the $\infty$-operad. 
We often will refer to an $\infty$-operad simply by way of its $\infty$-category of operators.
The definition implies that the fibers $\mathcal{O}([n])$ of an $\infty$-operad are canonically equivalent to the $n$-ary 
product $\prod_{1\leq i\leq n}\mathcal{O}([1])$; in particular, an element in $\mathcal{O}([n])$ is essentially an $n$-tuple of elements in
$\mathcal{O}([1])$. Given an $n$-tuple $\vec{x}\in\mathcal{O}([n])$ and a single object $y\in\mathcal{O}([1])$,
the definition further allows to think of the mapping space $\mathcal{O}(\vec{x},y)$ as the space of 
multi-morphisms from $\vec{x}$ to $y$ in the $\infty$-operad $\mathcal{O}$. For general objects $\vec{y}\in\mathcal{O}([m])$, the mapping 
space $\mathcal{O}(\vec{x},\vec{y})$ simply decomposes into the product of the mapping spaces $\mathcal{O}(\vec{x},y_i)$. In  
partiulcar, whenever the $\infty$-operad is single coloured, i.e.\ whenever $\mathcal{O}([1])$ is contractible with center of contraction 
$x$, then the mapping space $\mathcal{O}((x,\dots,x),x)$ may be thought of to consist of the $n$-ary operations associated to the
$\infty$-operad.

We further recall that a symmetric monoidal $\infty$-category $\mathcal{C}^{\otimes}\twoheadrightarrow\mathrm{Fin}_{\ast}$ is an
$\infty$-operad that exhibits cocartesian lifts to all morphisms in
$\mathrm{Fin}_{\ast}$ \cite[Example 2.1.2.18]{lurieha}. In particular, a symmetric monoidal $\infty$-category is a cocartesian fibration
$\mathcal{C}^{\otimes}\twoheadrightarrow\mathrm{Fin}_{\ast}$. By way of the $\infty$-categorical Grothendieck construction, it can be 
equivalently described as a Segal object $\mathrm{Fin}_{\ast}\rightarrow\mathrm{Cat}_{\infty}$ with a contractible $\infty$-category of 
objects. Essentially, that is, as an $\infty$-category object in $\mathrm{Cat}_{\infty}$ with a single object and commutative composition.
The underlying $\infty$-category of a symmetric monoidal $\infty$-category $\mathcal{C}^{\otimes}$ is given by the fiber
$\mathcal{C}:=\mathcal{C}^{\otimes}(\ast)$ over the unit $\ast\in\mathrm{Fin}_{\ast}$.

\paragraph{Algebras for $\infty$-operads}

Given an $\infty$-operad $\mathcal{O}$, the class of cocartesian morphisms in $\mathcal{O}$ over inert maps in $\mathrm{Fin}_{\ast}$ is 
called the class of inert morphisms in $\mathcal{O}$. Given another $\infty$-operad $\mathcal{V}$, one may consider the full
sub-$\infty$-category $\mathrm{Alg}_{\mathcal{O}}(\mathcal{V})\subseteq\mathrm{Fun}_{\mathrm{Fin}_{\ast}}(\mathcal{O},\mathcal{V})$ spanned 
by those morphisms of simplicial sets over $\mathrm{Fin}_{\ast}$ that preserve inert morphisms. This $\infty$-category is often referred to 
as the $\infty$-category of morphisms of $\infty$-operads from $\mathcal{O}$ to $\mathcal{V}$, or equivalently, as the $\infty$-category of 
$\mathcal{O}$-algebras in $\mathcal{V}$.

\begin{proposition}\label{prop_bv}
Let $\mathcal{C}^{\otimes}$ be a symmetric monoidal $\infty$-category and $\mathcal{V}$ be an $\infty$-operad. Then the $\infty$-category
$\mathrm{Alg}_{\mathcal{V}}(\mathcal{C}^{\otimes})$ inherits a symmetric monoidal structure from that of $\mathcal{C}^{\otimes}$.
\end{proposition}
\begin{proof}
\cite[Example 3.2.4.4]{lurieha}.
\end{proof}

\begin{example}\label{exple_bv}
Proposition~\ref{prop_bv} allows us to construct $\infty$-categories of algebras iteratively whenever we start with a symmetric monoidal
$\infty$-category $\mathcal{C}$. That is to say, given $\infty$-operads $\mathcal{O}$ and $\mathcal{V}$, we may construct
the $\infty$-category $\mathrm{Alg}_{\mathcal{O}}(\mathrm{Alg}_{\mathcal{V}}(\mathcal{C}^{\otimes})^{\otimes})$ of $\mathcal{O}$-algebras 
in $\mathcal{V}$-algebras in $\mathcal{C}^{\otimes}$. Moreover, the $\infty$-category of $\infty$-operads comes equipped with a closed 
symmetric monoidal structure itself \cite[Section 2.2.5]{lurieha}. The corresponding tensor product
$\mathcal{O}\otimes_{\mathrm{BV}}\mathcal{V}$ of two $\infty$-operads $\mathcal{O}$ and $\mathcal{V}$ is an $\infty$-operadic version of 
the Boardman-Vogt tensor product. It induces an equivalence
\[\mathrm{Alg}_{\mathcal{O}}(\mathrm{Alg}_{\mathcal{V}}(\mathcal{C}^{\otimes})^{\otimes})\simeq\mathrm{Alg}_{\mathcal{O}\otimes_{\mathrm{BV}}\mathcal{V}}(\mathcal{C}^{\otimes})\]
of symmetric monoidal $\infty$-categories.
\end{example}

\begin{example}
Central to this paper are the little cubes $\infty$-operads $\mathsf{E}_k$ for $0\leq k \leq \infty$. Given a 
symmetric monoidal $\infty$-category $\mathcal{C}^{\otimes}$, the $\infty$-category $\mathrm{Alg}_{\mathsf{E}_1}(\mathcal{C}^{\otimes})$ is 
(essentially) the $\infty$-category of associative and unital monoids in $\mathcal{C}^{\otimes}$. The $\infty$-operad $\mathsf{E}_{\infty}$ 
is given by the $\infty$-category $\mathrm{Fin}_{\ast}$ itself (or rather by the identity to itself). The $\infty$-category
$\mathrm{Alg}_{\mathsf{E}_{\infty}}(\mathcal{C}^{\otimes})$ is the $\infty$-category of commutative, associative and unital monoids in
$\mathcal{C}^{\otimes}$. We will say more about the intermediate integers $1<k<\infty$ below.
\end{example}

\begin{example}\label{exple_cartmon}
Every $\infty$-category $\mathcal{C}$ with finite products gives rise to a cartesian symmetric 
monoidal $\infty$-category $\mathcal{C}^{\times}\twoheadrightarrow\mathrm{Fin}_{\ast}$ \cite[Section 2.4.1]{lurieha}. In particular, there 
is a cartesian monoidal $\infty$-category $\mathrm{Cat}_{\infty}^{\times}$ of $\infty$-categories. Furthermore, every finite product 
preserving functor $F\colon\mathcal{C}\rightarrow\mathcal{D}$ between $\infty$-categories with finite products induces a symmetric monoidal 
functor $F\colon\mathcal{C}^{\times}\rightarrow\mathcal{D}^{\times}$ \cite[Corollary 2.4.1.8]{lurieha}.
\end{example}

\begin{example}
The $\infty$-category $\mathrm{SMonCat}_{\infty}:=\mathrm{Alg}_{\mathsf{E}_{\infty}}(\mathrm{Cat}_{\infty}^{\times})$ of
$\mathsf{E}_{\infty}$-algebras in 
the cartesian monoidal $\infty$-category $\mathrm{Cat}_{\infty}^{\times}$ is (equivalent to) the $\infty$-category of symmetric monoidal
$\infty$-categories. Indeed, its objects are the morphisms of $\infty$-operads from $\mathrm{Fin}_{\ast}$ to
$\mathrm{Cat}_{\infty}^{\times}$ by definition, and so they present symmetric monoidal
$\infty$-categories by way of the $\infty$-categorical Grothendieck construction. One may think of morphisms in $\mathrm{SMonCat}_{\infty}$ 
as the strong symmetric monoidal functors between symmetric monoidal $\infty$-categories. In contrast, general morphisms of
$\infty$-operads between symmetric monoidal $\infty$-categories are the lax symmetric monoidal functors.
\end{example}

\begin{example}\label{exple_moncat}
Similarly, the $\infty$-category $\mathrm{MonCat}_{\infty}:=\mathrm{Alg}_{\mathsf{E}_1}(\mathrm{Cat}_{\infty}^{\times})$ of
$\mathsf{E}_1$-monoids in the
cartesian monoidal $\infty$-category $\mathrm{Cat}_{\infty}^{\times}$ is (equivalent to) the $\infty$-category of monoidal
$\infty$-categories. The morphisms in $\mathrm{MonCat}_{\infty}=\mathrm{Alg}_{\mathsf{E}_1}(\mathrm{Cat}_{\infty}^{\times})$ are again the strong 
monoidal functors between monoidal $\infty$-categories (rather than their lax versions).
To relate the definition of a monoidal $\infty$-category to a notion of monoid in $\mathrm{Cat}_{\infty}$ more familiar to low dimensional 
category theory, we recall that there is an embedding $\Delta^{op}\hookrightarrow\mathrm{Fin}_{\ast}$ by assigning to a natural number 
$n$ its  associated ``cut'' \cite[Section 4.1.2]{lurieha}. Pullback of the cocartesian fibration
$\mathrm{Cat}_{\infty}^{\times}\twoheadrightarrow\mathrm{Fin}_{\ast}$ along this embedding yields a cocartesian fibration
$\mathrm{Cat}_{\infty}^{\times}\twoheadrightarrow\Delta^{op}$. This defines the ``planar'' monoidal $\infty$-category 
associated to $\mathrm{Cat}_{\infty}^{\times}$ by forgetting its symmetry. Now, the $\infty$-category $\mathrm{MonCat}_{\infty}$ is 
(equivalent to) the full sub-$\infty$-category
$\mathrm{Alg}_{\Delta^{op}}(\mathrm{Cat}_{\infty}^{\times})\subseteq\mathrm{Fun}_{\Delta^{op}}(\Delta^{op},\mathrm{Cat}_{\infty})$ spanned 
by the morphisms of planar $\infty$-operads. Again, the latter are defined as those sections
$\Delta^{op}\rightarrow\mathrm{Cat}_{\infty}^{\times}$ that preserve inert morphisms; or in this case, equivalently, as those functors
$\Delta^{op}\rightarrow\mathrm{Cat}_{\infty}$ that map $[0]$ to the terminal $\infty$-category, and that satisfy the Segal conditions. 
These are, essentially,  $\infty$-category objects in $\mathrm{Cat}_{\infty}$ with a single object.
\end{example}

Lastly, we note that if $\mathcal{C}^{\otimes}$ is a (not necessarily symmetric) monoidal $\infty$-category, then the $\infty$-category
$\mathrm{Alg}_{\mathsf{E}_1}(\mathcal{C}^{\otimes})$ of $\mathsf{E}_1$-algebras in $\mathcal{C}^{\otimes}$ can again be defined as the
$\infty$-category of 
morphisms from $\Delta^{op}$ to $\mathcal{C}^{\otimes}$ of planar $\infty$-operads \cite[Section 4.1.3]{lurieha}.\footnote{The notation
$\mathrm{Alg}_{\mathsf{E}_1}(\mathcal{C}^{\otimes})$ is not ideal in the non-symmetric context; the object is in this generality denoted by
$\mathrm{Alg}_{\mathsf{A}_{\infty}}(\mathcal{C}^{\otimes})$ instead. For the purposes of this paper it doesn't seem worth to introduce 
more notation at this point however.} For $k>1$, the $\infty$-category of $\mathsf{E}_k$-algebras in $\mathcal{C}^{\otimes}$ however is not a
well-typed construction if $\mathcal{C}^{\otimes}$ is merely associative and unital itself. Yet, more generally, there is a 
sequence of canonical inclusions
\[\mathsf{E}_1\hookrightarrow \mathsf{E}_2\hookrightarrow\dots\hookrightarrow \mathsf{E}_k\hookrightarrow\dots\hookrightarrow \mathsf{E}_{\infty}\]
which exhibits $\mathsf{E}_{\infty}$ as the sequential colimit of the $\mathsf{E}_k$'s \cite{lurieha}. Each $\mathsf{E}_k$ can hence be 
thought of as a finite approximation to $\mathsf{E}_{\infty}$.
In particular, whenever $\mathcal{C}^{\otimes}$ itself is an $\mathsf{E}_k$-algebra in $\mathrm{Cat}_{\infty}^{\times}$ for some
$1\leq k\leq \infty$, one may similarly consider $\mathcal{C}^{\otimes}$ as an $\infty$-category fibered over the
$\infty$-operad $\mathsf{E}_k$ (instead fibered over $\mathsf{E}_{\infty}=\mathrm{Fin}_{\ast}$, or over $\mathsf{E}_1\simeq\Delta^{op}$\footnote{This equivalence is to be understood as an equivalence in the sense of \cite[Theorem 4.1.3.14]{lurieha}.}). Then for any
$0\leq m	\leq k$ we again may define
$\mathrm{Alg}_{\mathsf{E}_m}(\mathcal{C}^{\otimes})\subseteq\mathrm{Fun}_{\mathsf{E}_k}(\mathsf{E}_m,\mathcal{C}^{\otimes})$ as the full 
sub-$\infty$-category spanned by those functors that preserve inert morphisms.

We end this preliminary section with a reminder of the $\infty$-operadic Additivity Theorem. We therefore refer back to the Boardman--Vogt tensor product of $\infty$-operads discussed in Example~\ref{exple_bv}.

\begin{theorem}[{\cite[Theorem 5.1.2.2]{lurieha}}]\label{thm_add}
Let $k_1,k_2\geq 0$ be integers. Then the canonical inclusions $E_{k_i}\hookrightarrow E_{k_1+k_2}$ induce an equivalence
$E_{k_1}\otimes_{\mathrm{BV}}E_{k_2}\rightarrow E_{k_1+k_2}$ of $\infty$-operads.
\end{theorem}\qed

The Additivity Theorem informally states that to give an $\mathsf{E}_k$-commutative algebra structure on an object $A$ in a symmetric 
monoidal $\infty$-category $\mathcal{C}^{\otimes}$ one may equivalently give $k$-many compatible associative algebra structures on $A$ in
$\mathcal{C}^{\otimes}$.

\subsection*{Monoidality of the $\infty$-category of $\mathsf{E}_k$-algebras}

To motivate Theorem~\ref{thm_main_12}, in this section we discuss a straightforward $\infty$-categorical version thereof 
(Corollary~\ref{corinftymain1}). In essence, it only uses the fact that the functor which assigns to an $\mathsf{E}_k$-monoidal
$\infty$-category $\mathcal{C}^{\otimes}$ the $\infty$-category $\mathrm{Alg}_{\mathsf{E}_m}(\mathcal{C}^{\otimes})$ of
$\mathsf{E}_m$-algebras therein is itself monoidal (for $0\leq m\leq k$).

Therefore, we recall that for any integer $0\leq k\leq\infty$ the $\infty$-operad $\mathsf{E}_k$ has a single object $u$. This object 
induces a morphism $\{u\}\colon\{\ast\}\rightarrow \mathsf{E}_k$ of simplicial sets over $\mathrm{Fin_{\ast}}$. Precomposition with $\{u\}$ 
induces a forgetful functor 
\begin{align}\label{equ_forgetful}
U\colon\mathrm{Alg}_{\mathsf{E}_k}(\mathcal{C}^{\otimes})\rightarrow\mathcal{C}
\end{align}
for any symmetric monoidal $\infty$-category $\mathcal{C}^{\otimes}$.

\begin{lemma}\label{lemma_algfinprod}
Let $\mathcal{C}^{\times}$ be a cartesian monoidal $\infty$-category, and let $0\leq k\leq\infty$ be an integer.
Then the $\infty$-category $\mathrm{Alg}_{\mathsf{E}_k}(\mathcal{C}^{\times})$ has again finite products, and the forgetful functor
\[\mathrm{Alg}_{\mathsf{E}_k}(\mathcal{C}^{\times})^{\times}\rightarrow\mathcal{C}^{\times}\]
is symmetric monoidal. In particular, this applies to
$\mathrm{MonCat}_{\infty}:=\mathrm{Alg}_{\mathsf{E}_1}(\mathrm{Cat}_{\infty}^{\times})$.
\end{lemma}
\begin{proof}
The forgetful functor $U\colon\mathrm{Alg}_{\mathsf{E}_k}(\mathcal{C}^{\times})\rightarrow\mathcal{C}$ creates all limits that exist in
$\mathcal{C}$ \cite[Corollary 3.2.2.5]{lurieha}. In particular, as $\mathcal{C}$ has finite products by assumption, this proves the 
statement (Example~\ref{exple_cartmon}).

\end{proof}


\begin{proposition}\label{thminftymain}
Let $0\leq k\leq \infty$. The functor
$\mathrm{Alg}_{\mathsf{E}_k}(-)\colon\mathrm{Alg}_{\mathsf{E}_k}(\mathrm{Cat}_{\infty})\rightarrow\mathrm{Cat}_{\infty}$ preserves finite 
products and hence gives rise to a symmetric monoidal functor
\[\mathrm{Alg}_{\mathsf{E}_k}(-)\colon\mathrm{Alg}_{\mathsf{E}_k}(\mathrm{Cat}_{\infty})^{\times}\rightarrow\mathrm{Cat}_{\infty}^{\times}.\]
\end{proposition}
\begin{proof}
The product of two $\mathsf{E}_k$-monoidal $\infty$-categories $\mathcal{C}^{\otimes}\twoheadrightarrow \mathsf{E}_k$,
$\mathcal{D}^{\otimes}\twoheadrightarrow \mathsf{E}_k$ is given by their fiber product
$\mathcal{C}^{\otimes}\times_{\mathsf{E}_k}\mathcal{D}^{\otimes}\twoheadrightarrow \mathsf{E}_k$. The functor
$\mathrm{Alg}_{\mathsf{E}_k}(-)$ is by definition the subfunctor of the global sections functor
$\mathrm{Fun}_{\mathsf{E}_k}(\mathsf{E}_k,-)$ that is pointwise spanned by those maps over $\mathsf{E}_k$ which preserve 
inert morphisms. The global sections functor $\mathrm{Fun}_{\mathsf{E}_k}(\mathsf{E}_k,-)$ preserves products (i.e.\ it maps pullbacks over 
$\mathsf{E}_k$ to products of hom-$\infty$-categories). A morphism in $\mathcal{C}^{\otimes}\times_{\mathsf{E}_k}\mathcal{D}^{\otimes}$ is 
inert if and only if its two projections in $\mathcal{C}^{\otimes}$ and $\mathcal{D}^{\otimes}$ are both inert. It follows that
$\mathrm{Alg}_{\mathsf{E}_k}(-)$ preserves finite products as well.
\end{proof}



\begin{corollary}\label{corinftymain1}
Let $0\leq k\leq\infty$ and suppose $\mathcal{C}^{\otimes}$ is an $\mathsf{E}_k$-monoidal $\infty$-category. Then for all $0\leq m\leq k$, 
the $\infty$-category $\mathrm{Alg}_{\mathsf{E}_m}(\mathcal{C}^{\otimes})$ of $\mathsf{E}_m$-monoids in $\mathcal{C}^{\otimes}$ is an
$E_{k-m}$-monoidal $\infty$-category.\footnote{For $k=\infty$, $0\leq m\leq \infty$, we again set $k - m =\infty$.}
\end{corollary}
\begin{proof}
Let $0\leq m\leq k<\infty$. By Proposition~\ref{thminftymain}, the functor
\[\mathrm{Alg}_{\mathsf{E}_m}(-)\colon\mathrm{Alg}_{\mathsf{E}_m}(\mathrm{Cat}_{\infty}^{\times})^{\times}\rightarrow\mathrm{Cat}_{\infty}^{\times}\]
is symmetric monoidal. Let $\mathcal{C}^{\otimes}$ be an $\mathsf{E}_k$-monoidal $\infty$-category. By way of the Additivity Theorem, we 
may consider $\mathcal{C}^{\otimes}$ as an $\mathsf{E}_{k-m}$-algebra in
$\mathrm{Alg}_{\mathsf{E}_m}(\mathrm{Cat}_{\infty}^{\times})^{\times}$. The composition
\[\mathsf{E}_{k-m}\xrightarrow{\mathcal{C}^{\otimes}}\mathrm{Alg}_{\mathsf{E}_m}(\mathrm{Cat}_{\infty}^{\times})^{\times}\xrightarrow{\mathrm{Alg}_{\mathsf{E}_m}(-)}\mathrm{Cat}_{\infty}^{\times}\]
of morphisms of $\infty$-operads hence defines an $\mathsf{E}_{k-m}$-monoidal structure on
$\mathrm{Alg}_{\mathsf{E}_m}(\mathcal{C}^{\otimes})$. For $k=\infty$ and any $0\leq m\leq\infty$, the statement follows directly from 
Proposition~\ref{prop_bv}.

%
\end{proof}

Let $U\colon\mathrm{MonCat}_{\infty}^{\times}\rightarrow\mathrm{Cat}_{\infty}^{\times}$ denote the forgetful functor that takes a 
monoidal $\infty$-category $\mathcal{C}^{\otimes}$ to its underlying $\infty$-category $\mathcal{C}$ (Lemma~\ref{lemma_algfinprod}). 
Then the forgetful functors (\ref{equ_forgetful}) can more generally be defined for (not necessarily symmetric) monoidal
$\infty$-categories $\mathcal{C}^{\otimes}$ whenever $k=1$. These give rise to a natural transformation
\begin{align*}
\pi\colon\mathrm{Alg}_{\mathsf{E}_1}(-)\rightarrow U
\end{align*}
of underlying functors given by evaluation at the object $u\in \mathsf{E}_1$. At a monoidal $\infty$-category $\mathcal{C}^{\otimes}$ it 
maps an $\mathsf{E}_1$-algebra in $\mathcal{C}^{\otimes}$ to its underlying base object. Analogously, there is a natural forgetful functor
\begin{align}\label{equ_forgetinfty}
\pi\colon\mathrm{Alg}_{\mathsf{E}_m}(-)\rightarrow U
\end{align}
from $\mathrm{Alg}_{\mathsf{E}_m}(-)\colon\mathrm{Alg}_{\mathsf{E}_k}(\mathrm{Cat}_{\infty}^{\times})\rightarrow\mathrm{Cat}_{\infty}$ to 
$U\colon\mathrm{Alg}_{\mathsf{E}_k}(\mathrm{Cat}_{\infty}^{\times})\rightarrow\mathrm{Cat}_{\infty}$ for all  $0\leq m\leq k$. 

\begin{corollary}\label{corinftymain2}
Let $0\leq k\leq\infty$ and suppose $\mathcal{C}^{\otimes}$ is an $\mathsf{E}_k$-monoidal $\infty$-category. Then for all $0\leq m\leq k$ 
the forgetful functor $\pi_{\mathcal{C}^{\otimes}}\colon\mathrm{Alg}_{\mathsf{E}_m}(\mathcal{C}^{\otimes})\rightarrow\mathcal{C}$ is
$\mathsf{E}_{k-m}$-monoidal with respect to the induced $\mathsf{E}_{k-m}$-algebra structure on
$\mathrm{Alg}_{\mathsf{E}_m}(\mathcal{C}^{\otimes})$ from Corollary~\ref{corinftymain1}.
\end{corollary}
\begin{proof}
Let $0\leq k\leq\infty$. The functors
$\mathrm{Alg}_{\mathsf{E}_m}(-)\colon\mathrm{Alg}_{\mathsf{E}_m}(\mathrm{Cat}_{\infty}^{\times})\rightarrow\mathrm{Cat}_{\infty}$ and 
$U\colon\mathrm{Alg}_{\mathsf{E}_m}(\mathrm{Cat}_{\infty}^{\times})\rightarrow\mathrm{Cat}_{\infty}$ are both finite product preserving.
The natural transformation (\ref{equ_forgetinfty}) hence lifts to a 2-cell 
\[\xymatrix{\mathrm{Alg}_{\mathsf{E}_m}(\mathrm{Cat}_{\infty}^{\times})^{\times}\ar@/^1pc/[r]^{\mathrm{Alg}_{\mathsf{E}_m}(-)}_{}\ar@/_1pc/[r]_{U}\ar@{}[r]|{\hspace{1pc}\Downarrow\pi} & \mathrm{Cat}_{\infty}^{\times}}\]
of symmetric monoidal functors over $\mathrm{Fin}_{\ast}$ by \cite[Corollary 2.4.1.8]{lurieha}. 
Let $\mathcal{C}^{\otimes}$ be an $\mathsf{E}_k$-monoidal $\infty$-category. As in the proof of Corollary~\ref{corinftymain1}, we
consider $\mathcal{C}^{\otimes}$ as an $\mathsf{E}_{k-m}$-algebra in
$\mathrm{Alg}_{\mathsf{E}_m}(\mathrm{Cat}_{\infty}^{\times})^{\times}$.
We recall that $\mathrm{Alg}_{\mathsf{E}_{k-m}}(-)\subseteq\mathrm{Fun}_{\mathrm{Fin}_{\ast}}(\mathsf{E}_{k-m},-)$ is defined to 
be pointwise a full subcategory. Thus, the $\mathsf{E}_{k-m}$-algebra $\mathcal{C}^{\otimes}$ gives rise to the 2-cell
\begin{align}\label{diag_corinftymain2}
\xymatrix{
\mathsf{E}_{k-m}\ar[r]^(.4){\mathcal{C}^{\otimes}} & \mathrm{Alg}_{\mathsf{E}_m}(\mathrm{Cat}_{\infty}^{\times})^{\times}\ar@/^1pc/[r]^{\mathrm{Alg}_{\mathsf{E}_m}(-)}_{}\ar@/_1pc/[r]_{U}\ar@{}[r]|{\hspace{1pc}\Downarrow\pi} & \mathrm{Cat}_{\infty}^{\times}}
\end{align}
of $\infty$-operads simply by way of precomposition. The top composition is precisely the $\mathsf{E}_{k-m}$-monoidal structure on
$\mathrm{Alg}_{\mathsf{E}_m}(\mathcal{C}^{\otimes})$ from Corollary~\ref{corinftymain1}. The bottom composition is equivalent to the 
canonical restriction
\[\mathsf{E}_{k-m}\rightarrow \mathsf{E}_k\xrightarrow{\mathcal{C}^{\otimes}}\mathrm{Cat}_{\infty}^{\times}\]
of morphisms of $\infty$-operads. The diagram (\ref{diag_corinftymain2}) hence defines an $\mathsf{E}_{k-m}$-monoidal structure on $\pi$ as 
stated. For $n=\infty$ and any $0\leq m\leq\infty$, the statement follows similarly. This case is in fact already explicitly captured by
\cite[Example 3.2.4.4]{lurieha}. 
\end{proof}

\begin{remark}
The proof of Corollary~\ref{corinftymain2} makes crucial use of the fact that $\mathrm{Cat}_{\infty}^{\times}$ is
cartesian. Indeed, given two symmetric monoidal functors $F,G\colon\mathcal{C}^{\otimes}\rightarrow\mathcal{D}{^\otimes}$ between 
general symmetric monoidal $\infty$-categories, it is not necessarily true that a 
natural transformation $\pi\colon F\rightarrow G$ between underlying functors automatically lifts to a symmetric 
monoidal natural transformation from $F$ to $G$ themselves. This is true however whenever $\mathcal{C}^{\otimes}$ and
$\mathcal{D}^{\otimes}$ are cartesian, because monoidality in this case is ensured by way of the universal property of finite products. 
\end{remark}

Informally, Corollary~\ref{corinftymain2} says that the monoidal structure on the $\infty$-category of $\mathsf{E}_m$-monoids in an
$\mathsf{E}_{k+m}$-monoidal $\infty$-category is directly induced from the pointwise monoidal structure on its underlying objects.
For Theorem~\ref{thm_main_12} we will not need this statement per se but rather a bicategorical version thereof. It will be proven 
in Corollary~\ref{corpsmonmonoidal} separately, however its proof is completely analogous.

\section{Monoidal categories and monoids therein}\label{sec_1cat}

To deduce Theorem~\ref{thm_main_12} directly from Corollary~\ref{corinftymain1}, we would want to embed the relevant fragment of the 
theory of monoidal bicategories in the theory of monoidal $\infty$-categories. In the special case of monoidal 1-categories, as to be 
illustrated in this section, this is fairly straightforward.

Therefore, we recall that the standard nerve functor $N\colon\mathrm{Cat}\rightarrow\mathrm{Set}^{\Delta^{op}}$ from the category of small 
categories into the category of simplicial sets
is the Quillen right adjoint of a homotopy localization with respect to the canonical model structure 
on $\mathrm{Cat}$ on the one hand, and the Joyal model structure for $\infty$-categories on $\mathrm{Set}^{\Delta^{op}}$ on the other hand. 
In particular, it induces a fully faithful and limit preserving functor
\begin{align}\label{equ_defN}
N\colon\mathrm{Ho}_{\infty}(\mathrm{Cat})\rightarrow\mathrm{Cat}_{\infty}
\end{align}
from the (locally groupoidal) $2$-category of small categories to the $\infty$-category of small $\infty$-categories.
As the $\infty$-categories $\mathrm{Ho}_{\infty}(\mathrm{Cat})$ and $\mathrm{Cat}_{\infty}$ are complete, we may consider both of
them as cartesian monoidal $\infty$-categories by Example~\ref{exple_cartmon}. We directly obtain the following ordinary categorical 
version of Theorem~\ref{thm_main_12}.

\begin{proposition}\label{prop_main1cat}
Let $\mathcal{C}^{\otimes}$ be an $\mathsf{E}_k$-algebra in the cartesian monoidal $\infty$-category
$\mathrm{Ho}_{\infty}(\mathrm{Cat})^\times$ of 
categories. Let $0\leq m\leq k$ be an integer. Then the $\infty$-category $\mathrm{Alg}_{\mathsf{E}_m}(N(\mathcal{C}^{\otimes}))$ of
$\mathsf{E}_m$-algebras in $N(\mathcal{C}^{\otimes})$ is again a category, and carries a canonical $\mathsf{E}_{k-m}$-algebra structure in
$\mathrm{Ho}_{\infty}(\mathrm{Cat})^\times$. Furthermore, the functor
\[\pi(N(\mathcal{C}^{\otimes}))\colon\mathrm{Alg}_{\mathsf{E}_m}(N(\mathcal{C}^{\otimes}))\rightarrow \mathcal{C}^{\otimes}\]
which assigns to an $\mathsf{E}_m$-algebra its underlying object lifts to a morphism of $\mathsf{E}_{k-m}$-algebras.
\end{proposition}
\begin{proof}
Let $\mathcal{C}^{\otimes}$ be an $\mathsf{E}_k$-algebra in $\mathrm{Ho}_{\infty}(\mathrm{Cat})^\times$. As the embedding (\ref{equ_defN}) 
is cartesian monoidal, it follows that $N(\mathcal{C}^{\otimes})$ is an $\mathsf{E}_k$-algebra in $\mathrm{Cat}_{\infty}^{\times}$. By 
Corollary~\ref{corinftymain1}, $\mathrm{Alg}_{\mathsf{E}_m}(N(\mathcal{C}^{\otimes}))$ is an $\mathsf{E}_{k-m}$-algebra in
$\mathrm{Cat}_{\infty}^{\times}$. Now, $\mathrm{Alg}_{\mathsf{E}_m}(N(\mathcal{C}^{\otimes}))$ is a full subcategory of the
$\infty$-category of functors into a category, and hence is also a category itself (i.e.\ it is equivalent to the nerve of a category). 
Consequently, the functor
$U\colon \mathrm{Alg}_{\mathsf{E}_m}(N(\mathcal{C}^{\otimes}))\rightarrow N(\mathcal{C})$ from Remark~\ref{corinftymain2} is a functor of 
categories as well. Furthermore, the embedding (\ref{equ_defN}) reflects $\mathsf{E}_m$-algebra structures. It follows that
$\mathrm{Alg}_{\mathsf{E}_m}(N(\mathcal{C}^{\otimes}))$ is an $\mathsf{E}_{k-m}$-algebra in $\mathrm{Ho}_{\infty}(\mathrm{Cat})^\times$, 
and that
$U\colon \mathrm{Alg}_{\mathsf{E}_m}(N(\mathcal{C}^{\otimes}))\rightarrow N(\mathcal{C}^{\otimes})$ is $\mathsf{E}_{k-m}$-monoidal as well. 
\end{proof}

To deduce an algebraic version of Proposition~\ref{prop_main1cat} akin to Theorem~\ref{mainthmgray}, the following algebraic notions of 
monoidal category theory refer to the ones of Joyal and Street \cite{joyalstreet}.

\begin{proposition}\label{prop_mon1cat_infty}
Let $\mathcal{C}$ be a category. Then there is an equivalence between the following notions:
\begin{enumerate}
\item monoidal structures on $\mathcal{C}$ and $\mathsf{E}_1$-algebra structures on $\mathcal{C}$ in $\mathrm{Ho}_{\infty}(\mathrm{Cat})^\times$;
\item braided monoidal structures on $\mathcal{C}$ and $\mathsf{E}_2$-algebra structures on $\mathcal{C}$ in $\mathrm{Ho}_{\infty}(\mathrm{Cat})^\times$;
\item symmetric monoidal structures on $\mathcal{C}$ and $\mathsf{E}_{\infty}$-algebra structures on $\mathcal{C}$ in
$\mathrm{Ho}_{\infty}(\mathrm{Cat})^\times$. 
\end{enumerate}
Furthermore, let $f\colon\mathcal{C}\rightarrow\mathcal{D}$ be a functor between categories $\mathcal{C}$ and $\mathcal{D}$.
\begin{enumerate}
\item Suppose $\mathcal{C}^{\otimes}$ and $\mathcal{D}^{\otimes}$ are monoidal. Then there is an equivalence between strong monoidal 
structures on $f$ and $\mathsf{E}_1$-structures on $f$ (considered as a morphism of associated $\mathsf{E}_1$-monoids).
\item Suppose $\mathcal{C}^{\otimes}$ and $\mathcal{D}^{\otimes}$ are braided monoidal. Then there is an equivalence between braided 
monoidal structures on $f$ and $E_2$-structures on $f$ (considered as a morphism of associated $\mathsf{E}_2$-monoids).
\item Suppose $\mathcal{C}^{\otimes}$ and $\mathcal{D}^{\otimes}$ are symmetric monoidal. Then there is an equivalence between symmetric monoidal structures on $f$ and $\mathsf{E}_{\infty}$-structures on $f$ (considered as a morphism of associated $\mathsf{E}_{\infty}$-monoids).
\end{enumerate}
\end{proposition}
\begin{proof}
The first triple of statements is \cite[Example 5.1.2.4]{lurieha}.\footnote{The argumentation in loc.\ cit.\ is somewhat brief but correct against the background of \cite[Propositions 5.2 - 5.4]{joyalstreet} (which in turn are somewhat brief but correct themselves).}
The second triple of statements follows along similar lines (Part (3) here is trivial).
\end{proof}

It follows that every monoidal category $\mathcal{C}^{\otimes}$ gives rise to an $\mathsf{E}_1$-monoidal $\infty$-category
$N(\mathcal{C}^{\otimes})$. Whenever the former is braided (symmetric), the latter can be extended to a corresponding
$\mathsf{E}_2$-monoid ($\mathsf{E}_{\infty}$-monoid) in $\mathrm{Cat}_{\infty}^{\times}$. For the following proposition, recall the 
definition of a (commutative)  monoid inside a (braided) monoidal category, as well as the definition of morphisms between such
\cite[Section 5]{joyalstreet}. To every monoidal category $\mathcal{C}^{\otimes}$ one may hence assign the category
$\mathrm{Mon}(\mathcal{C}^{\otimes})$ of monoids in $\mathcal{C}$. 

\begin{proposition}\label{prop_alg=hom_1}
Let $\mathcal{C}^{\otimes}$ be a monoidal category. 
\begin{enumerate}
\item The $\infty$-category $\mathrm{Alg}_{\mathsf{E}_1}(N(\mathcal{C}^{\otimes}))$ is naturally equivalent to the category
$\mathrm{Mon}(\mathcal{C}^{\otimes})$ of monoids in $\mathcal{C}^{\otimes}$.
\item If $\mathcal{C}^{\otimes}$ is braided, then the $\infty$-category $\mathrm{Alg}_{\mathsf{E}_2}(N(\mathcal{C}^{\otimes}))$ is 
naturally equivalent to the category $\mathrm{CMon}(\mathcal{C}^{\otimes})$ of symmetric (or commutative) monoids in
$\mathcal{C}^{\otimes}$.
\end{enumerate}
\end{proposition}
\begin{proof}
Part (1) follows from the observation that $\mathrm{Alg}_{\mathsf{E}_1}(N(\mathcal{C}^{\otimes}))$ is (up to natural equivalence) the 
category of functors $A\colon\Delta^{op}\rightarrow\mathcal{C}$ which exhibit $A$ as the Bar construction of a monoid $M$ in
$\mathcal{C}^{\otimes}$. Formally, this is given by \cite[Example 4.1.6.18]{lurieha} together with \cite[Theorem 5.4.3.5]{lurieha}.

For Part (2), if $\mathcal{C}^{\otimes}$ is braided monoidal,
we have equivalences
\[\mathrm{Alg}_{\mathsf{E}_2}(N(\mathcal{C}^{\otimes}))\simeq\mathrm{Alg}_{\mathsf{E}_1}(\mathrm{Alg}_{\mathsf{E}_1}(N(\mathcal{C}^{\otimes})))\simeq\mathrm{Alg}_{\mathsf{E}_1}(N(\mathrm{Mon}(\mathcal{C}^{\otimes})))\simeq\mathrm{Mon}(\mathrm{Mon}(\mathcal{C}^{\otimes}))\]
by Part (1) and the Additivity Theorem (Theorem~\ref{thm_add}). The latter is isomorphic to $\mathrm{CMon}(\mathcal{C}^{\otimes})$ by the 
classic Eckmann-Hilton argument.
\end{proof}

\begin{theorem}\label{cor_main1cat}
Let $\mathcal{C}^{\otimes}$ be a monoidal category.
\begin{enumerate}
\item Every braiding on $\mathcal{C}^{\otimes}$ induces a monoidal structure on $\mathrm{Mon}(\mathcal{C}^{\otimes})$.
\item Every symmetric braiding on $\mathcal{C}^{\otimes}$ induces \begin{enumerate}
\item a symmetric monoidal structure on $\mathrm{Mon}(\mathcal{C}^{\otimes})$,
\item a symmetric monoidal structure on the category $\mathrm{CMon}(\mathcal{C}^{\otimes})$. 
\end{enumerate}
\end{enumerate}
In each case the corresponding structure is canonically induced, i.e.\ the forgetful functors
$\mathrm{Mon}(\mathcal{C}^{\otimes})\rightarrow\mathcal{C}^{\otimes}$ and
$\mathrm{CMon}(\mathcal{C}^{\otimes})\rightarrow\mathcal{C}^{\otimes}$ are as monoidal as their domain is.
\end{theorem}
\begin{proof}
Let $\mathcal{C}^{\otimes}$ be a monoidal category considered as an $\mathsf{E}_1$-algebra in
$\mathrm{Ho}_{\infty}(\mathrm{Cat})^{\times}$. For Part (1), given a braiding $\beta$ on $\mathcal{C}$, let $\mathcal{C}^{\beta}$ 
be its associated $\mathsf{E}_1$-algebra in $\mathrm{Alg}_{\mathsf{E}_1}(\mathrm{Ho}_{\infty}(\mathrm{Cat})^{\times})$ with underlying
$\mathsf{E}_1$-algebra $\mathcal{C}^{\otimes}$ in $\mathrm{Ho}_{\infty}(\mathrm{Cat})^{\times}$. Then
$\mathrm{Alg}_{\mathsf{E}_1}(\mathcal{C}^{\beta})$ is an $\mathsf{E}_1$-algebra in $\mathrm{Ho}_{\infty}(\mathrm{Cat})^{\times}$ with underlying category $\mathrm{Alg}_{\mathsf{E}_1}(\mathcal{C}^{\otimes})$, and the forgetful functor
$U\colon\mathrm{Alg}_{\mathsf{E}_1}(\mathcal{C}^{\beta})\rightarrow\mathcal{C}^{\otimes}$ is a morphism of $\mathsf{E}_1$-algebras by 
Proposition~\ref{prop_main1cat}. By Proposition~\ref{prop_alg=hom_1} and Proposition~\ref{prop_main1cat}, that means that the category
$\mathrm{Mon}(\mathcal{C}^{\otimes})$ is monoidal, and that the forgetful functor
$\mathrm{Mon}(\mathcal{C}^{\beta})\rightarrow\mathcal{C}^{\otimes}$ is monoidal as well.

Part (2) is shown in the same way, additionally only using that symmetry is a fixed point in this context (which follows from the
Baez--Dolan--Lurie Stabilization Theorem \cite[Example 5.1.2.3]{lurieha}).
\end{proof}

\begin{remark}
Theorem~\ref{cor_main1cat} is well-known, albeit with a different proof; Part (2).(b) in particular can be found for example in
\cite[Section 5]{joyalstreet}.
\end{remark}

In (\ref{equ_intro_mono_mult}) we discussed a formula that defines a monoidal structure on $\mathrm{Mon}(\mathcal{C}^{\otimes})$ by hand 
whenever $\mathcal{C}^{\otimes}$ is a braided monoidal category. Here, the multiplication on the tensor product of two monoids is defined 
directly by way of the two component multiplications and the braiding on $\mathcal{C}^{\otimes}$. We end this section with the 
observation that the global constructions used to prove Theorem~\ref{cor_main1cat} faithfully recover these local constructions of
2-dimensional algebra as well. 

\begin{proposition}\label{prop_mult_explicit}
Let $\mathcal{C}^{\otimes}$ be a braided monoidal category. The multiplication on the tensor product of two monoids in
$\mathcal{C}^{\otimes}$ via the abstractly defined monoidal structure on $\mathrm{Mon}(\mathcal{C}^{\otimes})$ from
Theorem~\ref{cor_main1cat}.1 is precisely given by the composition (\ref{equ_intro_mono_mult}).
\end{proposition}
\begin{proof}
The assignment of a monoidal category $\mathcal{C}^{\otimes}$ to its associated category
$\mathrm{Mon}(\mathcal{C}^{\otimes})$ of monoids is functorial. For example, given a (strong) monoidal functor
$F\colon\mathcal{C}^{\otimes}\rightarrow\mathcal{D}^{\otimes}$ between monoidal categories, there is a natural isomorphism
\begin{align}\label{equ_1cat_chi_1}
\chi_F\colon F(-)\otimes F(-)\xrightarrow{\cong}F(-\otimes -)
\end{align}
as part of its data. The induced functor
$\mathrm{Mon}(F)\colon\mathrm{Mon}(\mathcal{C}^{\otimes})\rightarrow\mathrm{Mon}(\mathcal{D}^{\otimes})$ then maps a monoid $\vec{A}$ in
$\mathcal{C}^{\otimes}$ with multiplication $m_A\colon A\otimes A\rightarrow A$ to a monoid $F(\vec{A})$ with multiplication
\[F(A)\otimes F(A)\xrightarrow{\chi_F}F(A\otimes A)\xrightarrow{F(m_A)}A.\]
Essentially the same applies to monoidal $\infty$-categories and the $\infty$-category of $\mathsf{E}_1$-algebras therein. Indeed, let 
$F\colon\mathcal{C}^{\otimes}\rightarrow\mathcal{D}^{\otimes}$ be a monoidal functor between monoidal $\infty$-categories. 
Part of its data is a natural equivalence
\begin{align}\label{equ_1cat_chi_infty}
\chi_F\colon F(A)\otimes F(B)\xrightarrow{\simeq}F(A\otimes B)
\end{align}
for every pair of objects $A,B\in\mathcal{C}$. It is given by definition of the objects $A\otimes B$ by way of cocartesian lifts against the 
generic multiplication $d_1\colon [2]\rightarrow[1]$ in $\Delta^{op}$ with fixed domain $(A,B)\in\mathcal{C}^{\otimes}([2])$, and the 
assumption that $F$ preserves cartesian morphisms.\footnote{To make sense of this one has to pick distinguished representatives for the
the tensor product of a pair of objects in $\mathcal{C}^{\otimes}$. This generally can be done by Unstraightening the cocartesian fibration
$\mathcal{C}^{\otimes}\twoheadrightarrow\Delta^{op}$.} Let $\vec{A}$ be an $\mathsf{E}_1$-monoid in $\mathcal{C}^{\otimes}$, and let
$m_A\colon A\otimes A\rightarrow A$ be its associated multiplication. Then the cartesian monoidal functor
\begin{align}\label{equ_1cat_algfunctor}
\mathrm{Alg}_{\mathsf{E}_1}(-)\colon\mathrm{MonCat}_{\infty}^{\times}\rightarrow\mathrm{Cat}_{\infty}^{\times}
\end{align} 
from Proposition~\ref{thminftymain} induces a functor
\[\mathrm{Alg}_{v}(F)\colon\mathrm{Alg}_{\mathsf{E}_1}(\mathcal{C}^{\otimes})\rightarrow\mathrm{Alg}_{\mathsf{E}_1}(\mathcal{C}^{\otimes})\]
that maps $\vec{A}$ to an $\mathsf{E}_1$-monoid $F(\vec{A})$ in $\mathcal{D}^{\otimes}$ by way of post-composition of
$\vec{A}\colon \mathsf{E}_1\rightarrow\mathcal{C}^{\otimes}$ with $F$. In particular, its multiplication
$m_{F(A)}\colon F(A)\otimes F(A)\rightarrow F(A)$ is defined as the composition
\begin{align}\label{equ_inftycat_algmult}
\begin{gathered}
\xymatrix{
F(A)\otimes F(A)\ar[r]^{\chi_F}\ar@/_1pc/[dr]_{m_{F(A)}} & F(A\otimes A)\ar[d]^{F(m_A)} \\
 &  F(A).
}
\end{gathered}
\end{align}
For example, the functor (\ref{equ_1cat_algfunctor}) is cartesian monoidal and so it induces a functor
\begin{align}\label{equ_inftycat_algfunctor}
\mathrm{Alg}_{\mathsf{E}_1}(\mathrm{Alg}_{\mathsf{E}_1}(-))\colon\mathrm{Alg}_{\mathsf{E}_1}(\mathrm{MonCat}_{\infty}^{\times})\rightarrow\mathrm{Alg}_{\mathsf{E}_1}(\mathrm{Cat}_{\infty}^{\times}).
\end{align}
Let $\mathcal{C}^{\otimes}$ be an $\mathsf{E}_1$-algebra in $\mathrm{MonCat}_{\infty}$ --- or equivalently, an $\mathsf{E}_2$-algebra in
$\mathrm{Cat}_{\infty}^{\times}$ --- with multiplication
$-\otimes-\colon\mathcal{C}^{\otimes}\times\mathcal{C}^{\otimes}\rightarrow\mathcal{C}^{\otimes}$ in $\mathrm{MonCat}_{\infty}$.\footnote{Note that the underlying functor of $-\otimes-$ is necessarily the tensor product on $\mathcal{C}^{\otimes}$ by way of an
Eckmann-Hilton argument.}
The functor (\ref{equ_inftycat_algfunctor}) maps $\mathcal{C}^{\otimes}$ to a monoidal $\infty$-category
$\mathrm{Alg}_{\mathsf{E}_1}(\mathcal{C}^{\otimes})$ with a tensor product defined as the composition
\begin{align}\label{equ_inftycat_algmult_exple}
\begin{gathered}
\xymatrix{
\mathrm{Alg}_{\mathsf{E}_1}(\mathcal{C}^{\otimes})\times \mathrm{Alg}_{\mathsf{E}_1}(\mathcal{C}^{\otimes})\ar[rr]^(.55){\chi_{\mathrm{Alg}_{\mathsf{E}_1}(-)}}_(.55){\simeq}\ar@/_1pc/[drr]_{-\otimes_{\mathrm{Alg}_{\mathsf{E}_1}(\mathcal{C}^{\otimes})}-\hspace{2pc}} & & \mathrm{Alg}_{\mathsf{E}_1}(\mathcal{C}^{\otimes}\times\mathcal{C}^{\otimes})\ar[d]^{\mathrm{Alg}_{\mathsf{E}_1}(-\otimes-)} \\
 & & \mathrm{Alg}_{\mathsf{E}_1}(\mathcal{C}^{\otimes}).
}
\end{gathered}
\end{align}
Now, let $\vec{A},\vec{B}$ be two $\mathsf{E}_1$-monoids in $\mathcal{C}^{\otimes}$ with multiplications
$m_A\colon A\otimes A\rightarrow A$ and $m_B\colon B\otimes B\rightarrow B$, respectively. We want to compute (the multiplication of) their 
tensor product $\vec{A}\otimes \vec{B}$ in $\mathrm{Alg}_{\mathsf{E}_1}(\mathcal{C}^{\otimes})$. Therefore, we first note that the functor
\[\chi_{\mathrm{Alg}_{\mathsf{E}_1}(-)}\colon
\mathrm{Alg}_{\mathsf{E}_1}(\mathcal{C}^{\otimes})\times \mathrm{Alg}_{\mathsf{E}_1}(\mathcal{C}^{\otimes})\rightarrow\mathrm{Alg}_{\mathsf{E}_1}(\mathcal{C}^{\otimes}\times\mathcal{C}^{\otimes})
\]
in (\ref{equ_inftycat_algmult_exple}) maps a pair of $\mathsf{E}_1$-algebras $(\vec{A},\vec{B})$ in $\mathcal{C}^{\otimes}$ to the obvious
$\mathsf{E}_1$-algebra $\overrightarrow{(A,B)}$ in $\mathcal{C}^{\otimes}\times\mathcal{C}^{\otimes}$ with componentwise operations. In particular, its multiplication $m_{(A,B)}$ is the morphism
\[(A,B)\otimes (A,B):=(A\otimes A,B\otimes B)\xrightarrow{(m_A,m_B)} (A,B).\]
We are left to compute the multiplication of the $\mathsf{E}_1$-monoid $\mathrm{Alg}_{\mathsf{E}_1}(-\otimes-)(\overrightarrow{(A,B)})$. 
However, $-\otimes-$ itself is just a morphism in a monoidal $\infty$-category. Thus, according to (\ref{equ_inftycat_algmult}), it is 
given by the composition
\[\xymatrix{
(A\otimes B)\otimes (A\otimes B)\ar[r]^{\chi_{-\otimes -}}\ar@/_1pc/[dr]_{m_{\mathrm{Alg}_{\mathsf{E}_1}(-\otimes-)(\overrightarrow{(A,B)})}\hspace{2pc}} & (A\otimes A)\otimes (B\otimes B)\ar[d]^{m_A\otimes m_B} \\
 &  A\otimes B.
}
\]
Lastly, let us apply this to the case that we start with a braided monoidal category. We have seen in Proposition~\ref{prop_main1cat} that 
the functor (\ref{equ_1cat_algfunctor}) restricts to a functor
\[\mathrm{Alg}_{\mathsf{E}_1}(-)\colon\mathrm{Alg}_{\mathsf{E}_1}(\mathrm{Ho}_{\infty}(\mathrm{Cat}))^{\times}\rightarrow\mathrm{Ho}_{\infty}(\mathrm{Cat}))^{\times}.\]
Furthermore, whenever $F\colon\mathcal{C}^{\otimes}\rightarrow\mathcal{D}^{\otimes}$ is a monoidal functor between monoidal categories, 
then the natural isomorphism (\ref{equ_1cat_chi_1}) is precisely the natural equivalence (\ref{equ_1cat_chi_infty}) when considered as a 
morphism in $\mathrm{Alg}_{\mathsf{E}_1}(\mathrm{Cat})^{\times}$ via Proposition~\ref{prop_mon1cat_infty}. Now, if $\mathcal{C}^{\otimes}$ 
is a monoidal category equipped with a braiding $\beta$, then the tensor product
$\mathcal{C}^{\otimes}\times\mathcal{C}^{\otimes}\rightarrow\mathcal{C}^{\otimes}$ can be promoted to a monoidal functor
\cite[Proposition 5.2]{joyalstreet}.
Its associated natural isomorphism $\chi_{-\otimes -}$ in (\ref{equ_1cat_chi_1}) is given by the natural isomorphism 
$1\otimes\beta\otimes 1$. In particular, let $\vec{A},\vec{B}$ be two monoids in $\mathcal{C}^{\otimes}$. Under the identification of 
Proposition~\ref{prop_alg=hom_1}, the multiplication on the tensor product $\vec{A}\otimes\vec{B}$ in
$\mathrm{Alg}_{\mathsf{E}_1}(\mathcal{C}^{\otimes})$ is given by the composition
\[\xymatrix{
(A\otimes B)\otimes (A\otimes B)\ar[r]^{1\otimes\beta\otimes 1}\ar@/_1pc/[dr]_{m_{\mathrm{Alg}_{\mathsf{E}_1}(-\otimes-)(\overrightarrow{(A,B)})}\hspace{2pc}} & (A\otimes A)\otimes (B\otimes B)\ar[d]^{m_A\otimes m_B} \\
 &  A\otimes B.
}
\]
This is precisely the formula (\ref{equ_intro_mono_mult}) as stated.
\end{proof}

\section{Monoidal bicategories and monoids therein}\label{sec_modbicat}

In contrast to a category, a general bicategory $\mathcal{C}$ has non-invertible 2-cells and hence does not reduce to an
$\infty$-categorical structure without further ado. In particular, to equip a general bicategory $\mathcal{C}$ with a monoidal 
structure requires more than to equip its underlying locally groupoidal bicategory with one. We hence cannot apply 
Corollary~\ref{corinftymain1} directly. However, we can treat the homotopy theory of bicategories abstractly as a cartesian monoidal
$\infty$-category, and prove a suitable version of Corollary~\ref{corinftymain1} instead.

Therefore, in this section we define an $\infty$-categorical functor $\mathrm{Alg}_{\mathsf{E}_1}(-)$ that maps a monoidal bicategory
$\mathcal{C}$ to its bicategory of monoids by hand, and show that this functor is cartesian monoidal. This allows us to argue as in 
Corollary~\ref{corinftymain1} to prove Theorem~\ref{thm_main_12} directly. To manage the book-keeping efficiently, we work mostly model 
categorically. Therefore, we first note that both Theorem~\ref{thm_main_11} and Theorem~\ref{thm_main_12} can be restated in terms of 
Gray monoids. We recall that a Gray monoid is a (strict) monoid in the monoidal category
$\mathrm{Gray}:=(\twocat,\otimes_{\mathrm{Gr}})$ of 2-categories and (strict) 2-functors equipped with the Gray tensor product. As 
such, every Gray monoid is in particular a monoidal bicategory. Vice versa, there is the following theorem. 

\begin{theorem}\label{thm_bito2cat}
Every monoidal bicategory $\mathcal{C}$ is monoidally biequivalent to a Gray monoid $\bar{\mathcal{C}}$. 
By way of this biequivalence, every monoid in $\mathcal{C}$ gives rise to an essentially unique pseudomonoid in $\bar{\mathcal{C}}$ in the 
sense of \cite{daystreet}, and vice versa. 
\end{theorem}
\begin{proof}
The first statement is a popular special case of the coherence theorem for tricategories \cite{gps_coherence}. The second statement can for 
instance be derived via \cite{lackpsm}.
\end{proof}

Theorem~\ref{thm_bito2cat} allows us to reduce the proof of the theorems in the Introduction to statements about
$2$-categories and $2$-functors between them. Similarly, the higher algebra in the $\infty$-category of bicategories can be reduced to the higher algebra in the $\infty$-category of 2-categories as well by the following theorem of \cite{lack2catms}.

\begin{theorem}\label{thm_main_11}
The canonical inclusion $\twocat\rightarrow\mathrm{BiCat}$ is a right Quillen functor of associated model categories which induces an 
equivalence $\mathrm{Ho}_{\infty}(\twocat)\simeq\mathrm{Ho}_{\infty}(\mathrm{BiCat})$ of homotopy $\infty$-categories.\qed
\end{theorem}

\begin{remark}\label{rem_defbicatmon}
Although for a given monoidal bicategory $\mathcal{C}$, the notion of a pseudomonoid therein is a natural notion to consider, it appears 
that neither the bicategory $\mathrm{Mon}(\mathcal{C})$ nor an associated trifunctor $\mathrm{Mon}(-)$ have a standard definition in the 
literature. We will hence take the freedom to define it in the way most suitable for our purposes and adopt the 2-categorical point of view 
of \cite{daystreet}. That, is, we define $\mathrm{Mon}(\mathcal{C}):=\mathrm{PsMon}(\bar{\mathcal{C}})$ the category of pseudomonoids in 
its associated Gray monoid. The latter will be defined below.
\end{remark}

\subsection*{The homotopy theory of Gray monoids}

In this section we consider the category $2$-Cat equipped with its canonical model structure introduced by Lack \cite{lack2catms}. 
We recall from loc.\ cit.\ that $\mathrm{Gray}=(\twocat,\otimes_{\mathrm{Gr}})$ becomes a monoidal model category in this way.
Let $\mathrm{GrMon}$ denote the category of Gray monoids and (strict) Gray monoidal 2-functors. It arises as the category of algebras for 
the free monoid monad on $\mathrm{Gray}$ and hence comes equipped with a (monadic) adjunction as follows.
\begin{align}\label{equ_propgrmmodstr}
\xymatrix{
\mathrm{GrMon}\ar@<-.5ex>[r]_(.55)U & \twocat\ar@<-.5ex>[l]_(.45)F
}
\end{align}

\begin{proposition}\label{propgrmmodstr}
The adjunction (\ref{equ_propgrmmodstr}) allows to 
transfer the canonical model structure on $\twocat$ to a model structure on $\mathrm{GrMon}$ such that a morphism $f$ in
$\mathrm{GrMon}$ is a (trivial) fibration if and only if $U(f)$ is so in $\twocat$.

\end{proposition}
\begin{proof}
To construct this model structure it suffices to show that the monoidal model 
category $\mathrm{Gray}$ satisfies the monoid axiom \cite[Definition 3.3, Theorem 4.1]{schwedeshipley_mon}.
This is shown in \cite[Theorem 7.7]{lack2catms}. The second statement follows directly from \cite[Theorem 4.1.(3)]{schwedeshipley_mon}.

\end{proof}

\begin{remark}
The model structure on $\mathrm{GrMon}$ from Proposition~\ref{propgrmmodstr} is the single object version of Lack's model structure on
the category of Gray categories \cite{lackgraycatms}. 
\end{remark}

We hence may assign a homotopy $\infty$-category $\mathrm{Ho}_{\infty}(\mathrm{GrMon})$ and $\mathrm{Ho}_{\infty}(\twocat)$ 
to the two corresponding model categories each. Next, we want to define a symmetric monoidal functor

\begin{align}\label{equdefpsmoninfty}
\mathrm{Alg}_{\mathsf{E}_1}(-)\colon\mathrm{Ho}_{\infty}(\mathrm{GrMon})^{\times}\rightarrow\mathrm{Ho}_{\infty}(\twocat)^{\times}
\end{align}
that assigns to a Gray monoid the 2-category of pseudomonoids in it.

\begin{remark}\label{rem_psmon}
Ideally, we would like to use the fact that Gray is closed monoidal to define the $\infty$-categorical functor 
(\ref{equdefpsmoninfty}) entirely by the means of abstract homotopy theory. For instance, let us recall Lack's ``coherent'' 
characterization of pseudomonoids in a Gray monoid as Gray functors out of the walking pseudomonoid $\mathrm{PsM}$ \cite{lackpsm}. Thus, 
if the model category $\mathrm{GrMon}$ was $\mathrm{Gray}$-enriched, we could define the functor (\ref{equdefpsmon}) simply as the
Gray-enriched functor corepresented by $\mathrm{PsM}$. Indeed, it is not hard to show that the Gray monoid $\mathrm{PsM}$ is cofibrant, and 
so its associated corepresentable would be a Quillen right adjoint automatically. This would essentially finish the discussion. However, 
the (model) category $\mathrm{GrMon}$ appears to fail to be Gray-enriched. This failure boils down to the fact that 
the functorial diagonal $\mathcal{C}\rightarrow\mathcal{C}\times\mathcal{C}$ of the cartesian structure on $\twocat$ only lifts to a 
pseudofunctor $\mathcal{C}\rightarrow\mathcal{C}\otimes_{\mathrm{Gr}}\mathcal{C}$ 
rather than to an actual 2-functor. One can hence make precise the fact that the model category $\mathrm{GrMon}$ is weakly enriched in
$\mathrm{Gray}$ instead; this however complicates matters rather than to simplify them. We hence will essentially follow the same playbook, 
but in a little more hands-on fashion.

\end{remark}

In the following, we define the 2-category $\mathrm{PsMon}(\mathcal{C})$ of pseudomonoids in a Gray monoid $\mathcal{C}$. Its objects are 
the pseudomonoids as defined in \cite[Section 3]{daystreet}. Here, given two morphisms $f_i\colon C_i\rightarrow D_i$, $i=0,1$ in a Gray 
monoid $\mathcal{C}$ with multiplication $\otimes\colon\mathcal{C}\otimes_{\mathrm{Gr}}\mathcal{C}\rightarrow\mathcal{C}$, we adopt 
the convention to denote by $f\otimes g\colon C_1\otimes C_2\rightarrow D_1\otimes D_2$ the composition $(f\otimes 1)\circ(1\otimes g)$ 
(rather than the non-equal but isomorphic choice $(1\otimes g)\circ(f\otimes 1)$).

\begin{definition}\label{def_psmonobj}
Let $\mathcal{C}$ be a Gray monoid with multiplication $\otimes\colon\mathcal{C}\otimes_{\mathrm{Gr}}\mathcal{C}$ and unit
$I\in\mathcal{C}$. The 2-category $\mathrm{PsMon}(\mathcal{C})$ of pseudomonoids in $\mathcal{C}$ is given by the following data (see
e.g.\ \cite[Section 2]{mccrudden} for more details).
\begin{enumerate}
\item Its objects are the pseudomonoids in $\mathcal{C}$; that is, objects $A\in\mathcal{C}$ together with morphisms
$m\colon A\otimes A\rightarrow A$ and $\eta\colon I\rightarrow A$ in $\mathcal{C}$, as well as invertible 2-cells
\begin{align*}
\xymatrix{
A\otimes A\otimes A\ar[r]^{1\otimes m}\ar[d]_{m\otimes 1}\ar@{}[dr]|{\overset{a}{\Rightarrow}} & A\otimes A\ar[d]^m\\
A\otimes A\ar[r]_{m} & A
} & & 
\xymatrix{
A\ar[r]^{\eta\otimes 1}\ar[dr]_1\ar@{}[dr]^{\overset{l}{\Leftarrow}} & A\otimes A\ar[d]|m & A\ar[l]_{1\otimes\eta}\ar[dl]^{1}\ar@{}[dl]_{\overset{r}{\Rightarrow}} \\
& A &
}
\end{align*}	
such that two equations hold that express higher associativity and unitality of the associator $a$.
\item Given two pseudomonoids $\vec{A}=(A,m_A,\eta_A,a_A,l_A,r_A)$ and $\vec{B}=(B,m_B,\eta_B,a_B,l_B,r_B)$, a morphism
$\vec{f}\colon\vec{A}\rightarrow\vec{B}$ is a morphism $f\colon A\rightarrow B$ in $\mathcal{C}$ together with invertible 2-cells
\begin{align*}
\xymatrix{
A\otimes A\ar[r]^{f\otimes f}\ar[d]_{m_A}\ar@{}[dr]|{\overset{\chi_f}{\Rightarrow}} & B\otimes B\ar[d]^{m_B} \\
A\ar[r]_f & B
} & & 
\xymatrix{
I\ar[r]^{\eta_A}\ar[dr]_{\eta_B}\ar@{}[dr]^{\overset{\iota_f}{\Rightarrow}} & A\ar[d]^f \\
 & B
}
\end{align*}
subject to coherence conditions.
\item Given two morphisms $\vec{f}=(f,\chi_f,\iota_f)$ and $\vec{g}=(g,\chi_g,\iota_g)$ from $\vec{A}$ to $\vec{B}$, a 2-cell 
$\alpha\colon\vec{f}\rightarrow\vec{g}$ is a 2-cell $\alpha\colon f\rightarrow g$ in $\mathcal{C}$ subject to coherence conditions.
\end{enumerate}
Vertical and horizontal composition of cells as well as identities are directly induced from the vertical and horizontal composition (and 
their identities) in $\mathcal{C}$.
\end{definition}

\begin{remark}
Definition~\ref{def_psmonobj} is also the single object version of the Gray category $\mathbf{PsMnd}(\mathcal{K})$ of pseudomonads in a 
Gray category $\mathcal{K}$ defined by Gambino and Lobbia \cite[Theorem 2.5]{gambinolobbia}. Hence, the fact that
$\mathrm{PsMon}(\mathcal{C})$ forms a 2-category for any Gray monoid $\mathcal{C}$ can be directly derived from loc.\ cit\ as well. 
Therefore, note that when we consider a given Gray monoid $\mathcal{C}$ as a Gray category with a single object $\ast$, then
$\mathrm{PsMon}(\mathcal{C})\subseteq\mathbf{PsMnd}(\mathcal{C})$ is spanned by all pseudomonads on $\ast$
\cite[Definition 1.5]{gambinolobbia}, all pseudomonad morphisms $(F,\phi)$ where $F$ is the identity $1_{\ast}$ on $\ast$
\cite[Definition 2.1]{gambinolobbia}, all pseudomonad transformations $(p,\bar{p})$ where $p$ is the identity on $1_{\ast}$ 
\cite[Definition 2.3]{gambinolobbia}, and all pseudomonad modifications $\alpha$ which themselves are the identity $1_{1_{\ast}}$ 
\cite[Definition 2.4]{gambinolobbia}. The latter are automatically trivial.
\end{remark}

Let $\mathcal{C}$ and $\mathcal{D}$ be Gray monoids. Every Gray functor $f\colon\mathcal{C}\rightarrow\mathcal{D}$ induces a 2-functor
$\mathrm{PsMon}(\mathcal{C})\rightarrow\mathrm{PsMon}(\mathcal{D})$. This action defines a functor 
\begin{align}\label{equdefpsmon}
\mathrm{PsMon}(-)\colon\textrm{GrMon}\rightarrow \twocat.
\end{align}

\begin{proposition}\label{prop_psmonrqui}
The functor (\ref{equdefpsmon}) preserves trivial fibrations whenever $\mathrm{GrMon}$ is equipped with the model 
structure from Proposition~\ref{propgrmmodstr}. It furthermore preserves finite products.
\end{proposition}
\begin{proof}
By \cite{lack2catms}, the trivial fibrations in $\twocat$ are precisely the 2-functors that are surjective on objects and locally 
surjective-on-objects equivalences of hom-categories. By Proposition~\ref{propgrmmodstr}, the trivial fibrations in $\mathrm{GrMon}$ are 
precisely those Gray-functors whose underlying 2-functor is a trivial fibration of 2-categories. Thus, let
$p\colon\mathcal{C}\rightarrow\mathcal{D}$ be a trivial fibration of Gray monoids. To show that
$p_{\ast}\colon\mathrm{PsMon}(\mathcal{C})\rightarrow\mathrm{PsMon}(\mathcal{D})$ is again a trivial fibration of 2-categories, we are to 
show, first, that every pseudomonoid $A$ in $\mathcal{D}$ can be lifted to a pseudomonoid $\bar{A}$ in $\mathcal{C}$ such that
$p_{\ast}(\bar{A})=A$, second, that every morphism $f\colon A\rightarrow B$ of pseudomonoids in $\mathcal{D}$ can be lifted to a morphism 
$\bar{f}\colon \bar{A}\rightarrow\bar{B}$ of pseudomonoids in $\mathcal{C}$ such that $p_{\ast}(\bar{f})=f$, and, third, that $p_{\ast}$ 
is locally fully faithful. All three statements are a straightforward verification.

The fact that $\mathrm{PsMon}(-)$ preserves finite products is similarly straightforward, given that all relevant operations are defined 
componentwise.
\end{proof}

%


\begin{remark}
We note that the Gray tensor product on $\twocat$ can be lifted to a symmetric monoidal structure on $\mathrm{GrMon}$ by  Part 2.(a) of
Corollary~\ref{cor_main1cat}. The fact that the functor (\ref{equdefpsmon}) preserves finite products as stated in 
Proposition~\ref{prop_psmonrqui} is a higher categorical adjustment of the apparent failure of
$\mathrm{PsMon}(-)\colon(\textrm{GrMon},\otimes_{\mathrm{Gr}})\rightarrow(\twocat,\otimes_{\mathrm{Gr}})$ to be
symmetric monoidal, or even lax monoidal for that matter. Indeed, we only have a zig-zag
\[\mathrm{PsMon}(\mathcal{C}\otimes_{\mathrm{Gr}}\mathcal{D})\xrightarrow{\sim}\mathrm{PsMon}(\mathcal{C}\times\mathcal{D})\cong\mathrm{PsMon}(\mathcal{C})\times\mathrm{PsMon}(\mathcal{D})\xleftarrow{\sim}\mathrm{PsMon}(\mathcal{C})\otimes_{\mathrm{Gr}}\mathrm{PsMon}(\mathcal{D})\]
of natural weak equivalences rather than a single arrow. Here, the outer two weak equivalences exist by Proposition~\ref{prop_psmonrqui} 
together with the fact that all 2-categories are fibrant, and that there is a natural pointwise trivial fibration
$\otimes_{\mathrm{Gr}}\rightarrow\times$. Indeed, much of the discussion in this section is owed to the fact that the Gray tensor product
$\otimes_{\mathrm{Gr}}$ is not cartesian, but naturally weakly equivalent to the cartesian product. This makes the constructions we aim to 
consider only homotopy-coherent rather than strict. This in turn makes it natural to work in the underlying $\infty$-categories.

\end{remark}

\begin{corollary}\label{cor_inftypsmon}
The functor $\mathrm{PsMon}(-)\colon\mathrm{GrMon}\rightarrow \twocat$ descends to a cartesian monoidal functor
\begin{align}\label{equ_inftypsmon}
\mathrm{Alg}_{\mathsf{E}_1}(-):=\mathrm{Ho}_{\infty}(\mathrm{PsMon}(-))\colon\mathrm{Ho}_{\infty}(\mathrm{GrMon})^{\times}\rightarrow\mathrm{Ho}_{\infty}(\twocat)^{\times}
\end{align}
on underlying $\infty$-categories.
\end{corollary}
\begin{proof}
The functor $\mathrm{PsMon}(-)\colon\textrm{GrMon}\rightarrow \twocat$
preserves trivial fibrations by Proposition~\ref{prop_psmonrqui}. As all objects in $\mathrm{GrMon}$ are fibrant, by Brown's Lemma 
\cite[Lemma 1.1.12]{hovey} it hence preserves all weak equivalences. The functor $\mathrm{PsMon(-)}$ thus induces a right derived functor 
(\ref{equ_inftypsmon}) of $\infty$-categories as stated. Binary products (and hence all finite products) in both
$\mathrm{Ho}_{\infty}(\mathrm{GrMon})$ and $\mathrm{Ho}_{\infty}(\twocat)$ can be computed as corresponding binary ordinary 
categorical products (of fibrant objects) in $\mathrm{GrMon}$ and $\twocat$, respectively. Such are preserved by $\mathrm{PsMon}(-)$ 
by Proposition~\ref{prop_psmonrqui}.
\end{proof}

We are left to relate the abstractly defined $\infty$-category $\mathrm{Ho}_{\infty}(\mathrm{GrMon})$ to the $\infty$-category of 
$\mathsf{E}_1$-algebras in the cartesian monoidal $\infty$-category $\mathrm{Ho}_{\infty}(\twocat)^{\times}$. This is precisely 
given by the following lemma.

\begin{lemma}\label{lemma_grmon=E1}
The forgetful functor $U\colon\mathrm{GrMon}\rightarrow \twocat$ induces a right adjoint
\[\mathrm{Ho}_{\infty}(U)\colon\mathrm{Ho}_{\infty}(\mathrm{GrMon})\rightarrow\mathrm{Ho}_{\infty}(\twocat)\]
that lifts to an equivalence
\[\mathrm{Ho}_{\infty}(U)\colon\mathrm{Ho}_{\infty}(\mathrm{GrMon})\rightarrow \mathrm{Alg}_{\mathsf{E}_1}(\mathrm{Ho}_{\infty}(\twocat)^\times).\]
\end{lemma}
\begin{proof}
The monoidal structure on $\mathrm{Ho}_{\infty}(\twocat)$ induced by the monoidal model category $\mathrm{Gray}$ is the cartesian 
product by way of the natural pointwise trivial fibration $\otimes_{\mathrm{Gr}}\rightarrow \times$ in $\twocat$.
Thus, the statement follows directly from the rectification theorem of \cite[Theorem 4.1.8.4]{lurieha} since all conditions are satisfied.

\end{proof}


\subsection*{$\mathsf{E}_k$-algebras in the $\infty$-category of 2-categories}

By way of Corollary~\ref{cor_inftypsmon} and Lemma~\ref{lemma_grmon=E1} we have constructed a functor
\[\mathrm{Alg}_{\mathsf{E}_1}(-)\colon\mathrm{Alg}_{\mathsf{E}_1}(\mathrm{Ho}_{\infty}(\twocat)^\times)\rightarrow\mathrm{Ho}_{\infty}(\twocat)^\times\]
of $\infty$-categories that maps an $\mathsf{E}_1$-algebra $A$ -- when presented by a Gray monoid $\mathcal{C}$ -- to the 2-category
$\mathrm{Alg}_{\mathsf{E}_1}(A):=\mathrm{PsMon}(\mathcal{C})$ of pseudomonoids in $\mathcal{C}$.

\begin{corollary}\label{corpsmonmonoidal}
Every $\mathsf{E}_{k+1}$-monoid structure on a 2-category $\mathcal{C}$ in $\mathrm{Ho}_{\infty}(\twocat)^\times$ gives rise to an
$\mathsf{E}_k$-monoid structure on the 2-category $\mathrm{PsMon}(\mathcal{C})$ of pseudomonoids in $\mathcal{C}$ in
$\mathrm{Ho}_{\infty}(\twocat)^\times$.
\end{corollary}
\begin{proof}
The functor $\mathrm{Alg}_{\mathsf{E}_1}(-)$ is cartesian monoidal by Corollary~\ref{cor_inftypsmon} and hence lifts to a functor
\[\mathrm{Alg}_{\mathsf{E}_1}(-)\colon\mathrm{Alg}_{\mathsf{E}_k}(\mathrm{Alg}_{\mathsf{E}_1}(\mathrm{Ho}_{\infty}(\twocat)^\times))\rightarrow\mathrm{Alg}_{\mathsf{E}_k}(\mathrm{Ho}_{\infty}(\twocat)^\times).\]
The domain is $\mathrm{Alg}_{\mathsf{E}_{k+1}}(\mathrm{Ho}_{\infty}(\twocat)^\times)$ by the Additivity Theorem (Theorem~\ref{thm_add}).
\end{proof}

Furthermore, we note that the functor
\[\mathrm{PsMon}(-)\colon\textrm{GrMon}\rightarrow \twocat\]
comes equipped with a canonical natural transformation to the forgetful functor $U\colon\textrm{GrMon}\rightarrow \twocat$ from the 
adjunction (\ref{equ_propgrmmodstr}). Indeed, at any given Gray monoid $\mathcal{C}$, we may consider the natural 2-functor
\begin{align}\label{defnattransfmonobase}
\pi_{\mathcal{C}}\colon\mathrm{PsMon}(\mathcal{C})\rightarrow U(\mathcal{C})
\end{align}
that maps a pseudomonoid $\vec{A}\in\mathcal{C}$ to its underlying object $A\in\mathcal{C}$, a morphism
$\vec{f}\colon\vec{A}\rightarrow\vec{B}$ of pseudomonoids to its underlying morphism $f\colon A\rightarrow B$ of underlying objects, and 
a 2-cell $\alpha$ two itself.

\begin{remark}
On the level of objects, the natural transformation $\pi\colon\mathrm{PsMon}(-)\rightarrow U$ is given by precomposition with the
Gray-functor $\{[0]\}\colon\ast\rightarrow\mathrm{PsM}$.
\end{remark}

\begin{corollary}\label{cor_monobasemap}
The natural transformation  (\ref{defnattransfmonobase}) induces a cartesian monoidal natural transformation
\begin{align}\label{equ_monobasemap}
\pi\colon\mathrm{Alg}_{\mathsf{E}_1}(-)\rightarrow U
\end{align}
from the cartesian monoidal functor $\mathrm{Alg}_{\mathsf{E}_1}(-)\colon\mathrm{Alg}_{\mathsf{E}_1}(\mathrm{Ho}_{\infty}(\twocat)^{\times})^{\times}\rightarrow\mathrm{Ho}_{\infty}(\twocat)^{\times}$ to the cartesian monoidal functor
$U\colon\mathrm{Alg}_{\mathsf{E}_1}(\mathrm{Ho}_{\infty}(\twocat)^{\times})^{\times}\rightarrow\mathrm{Ho}_{\infty}(\twocat)^{\times}$.  
\end{corollary}
\begin{proof}
For any given combinatorial model category $\mathbb{M}$, the $\infty$-category $\mathrm{Ho}_{\infty}(\mathbb{M})^{\Delta^1}$ of arrows 
in $\mathrm{Ho}_{\infty}(\mathbb{M})$ is given by the homotopy $\infty$-category of the category $\mathbb{M}^{\Delta^1}$ of arrows in
$\mathbb{M}$ when equipped with any of the model structures with pointwise weak equivalences. In particular, we are to show that the 
natural transformation (\ref{defnattransfmonobase}) considered as a functor
\[\pi\colon\mathrm{GrMon}\rightarrow \twocat^{\Delta^1}\]
descends to a natural transformation on homotopy $\infty$-categories. Therefore it suffices to show that $\pi$ preserves weak equivalences 
pointwise, which is precisely given by the fact that both its domain $\mathrm{PsMon}(-)$ and its codomain $U$ preserve weak equivalences.

Furthermore, just as in the proof of Corollary~\ref{corinftymain2}, any natural transformation between two functors that each carry a 
cartesian monoidal structure can be lifted to an essentially unique cartesian monoidal (i.e.\ symmetric monoidal) natural transformation 
itself.
\end{proof}

\begin{corollary}\label{cor_monforget}
Suppose $\mathcal{C}$ is an $\mathsf{E}_{k+1}$-monoid in $\mathrm{Ho}_{\infty}(\twocat)^{\times}$ for some $k\geq 0$, and let 
$U(\mathcal{C})$ be its underlying $\mathsf{E}_k$-monoid. 
Let $\mathrm{PsMon}(\mathcal{C})$ be equipped with the associated $\mathsf{E}_k$-monoidal structure from Corollary~\ref{corpsmonmonoidal}. 
Then the morphism $\pi_{\mathcal{C}}\colon\mathrm{PsMon}(\mathcal{C})\rightarrow U(\mathcal{C})$ in $\mathrm{Ho}_{\infty}(\twocat)$ 
lifts to a morphism $\pi_{\mathcal{C}}\colon\mathrm{PsMon}(\mathcal{C})\rightarrow U(\mathcal{C})$ of $\mathsf{E}_k$-monoids.
\end{corollary}
\begin{proof}
This is an immediate formal consequence of Corollary~\ref{cor_monobasemap} (following the blueprint of Corollary~\ref{corinftymain2}).  
Indeed, the natural transformation (\ref{equ_monobasemap}) is a 2-cell between cocartesian functors of cocartesian fibrations over
$\mathrm{Fin}_{\ast}$. In particular, as $\mathrm{Alg}_{\mathsf{E}_k}(-)\subseteq\mathrm{Fun}_{\mathrm{Fin}_{\ast}}(\mathsf{E}_k,-)$ is 
defined to be pointwise a full subcategory, every $\mathsf{E}_k$-monoid
$\mathcal{C}\in\mathrm{Alg}_{\mathsf{E}_k}(\mathrm{Ho}_{\infty}(\mathrm{GrMon})^{\times})$ 
gives rise to a morphism $\pi_{\mathcal{C}}\colon\mathrm{PsMon}(\mathcal{C})\rightarrow U(\mathcal{C})$ of $\mathsf{E}_k$-monoids simply by 
way of precomposition with $\mathcal{C}$.
\end{proof}

Lastly,
suppose $\mathcal{C}$ is an $\mathsf{E}_k$-monoid in $\mathrm{Ho}_{\infty}(\twocat)^{\times}$ for some $k\geq 1$. Then
$\mathrm{Alg}_{\mathsf{E}_1}(\mathcal{C})$ is an $\mathsf{E}_{k-1}$-monoid in $\mathrm{Ho}_{\infty}(\twocat)^{\times}$ by 
Corollary~\ref{corpsmonmonoidal}
\footnote{The $E_0$-structure on $\mathrm{Alg}_{\mathsf{E}_1}(\mathcal{C})$ in case $k=1$ is fairly trivial.}. By way of the Additivity 
Theorem, we may hence iterate this construction $m$-times for any $m\leq k$ to obtain an $\mathsf{E}_{k-m}$-monoid
\[\mathrm{Alg}_{\mathsf{E}_m}(\mathcal{C}):=\mathrm{Alg}_{\mathsf{E}_1}^{(m)}(\mathcal{C})\]
of $m$-fold pseudomonoids in $\mathcal{C}$. A specification of this object in more concrete algebraic terms will be addressed in 
Conjecture~\ref{conj_higherEnint}.


\begin{theorem}\label{thm_main_12}
Let $\mathcal{C}$ be an $\mathsf{E}_k$-monoid in $\mathrm{Ho}_{\infty}(\text{2-Cat})^\times$. Let $0\leq m\leq k$ be an integer. Then the
2-category $\mathrm{Alg}_{\mathsf{E}_m}(\mathcal{C})$ carries a canonical $\mathsf{E}_{k-m}$-monoid structure. Furthermore, the 2-functor
\[U\colon\mathrm{Alg}_{\mathsf{E}_m}(\mathcal{C})\rightarrow \mathcal{C}\]
which assigns to an $m$-fold pseudomonoid its underlying object lifts to a morphism of $\mathsf{E}_{k-m}$-monoids.
\end{theorem}
\begin{proof}
Immediate by Corollary~\ref{corpsmonmonoidal} and Corollary~\ref{cor_monforget}.
\end{proof}

\begin{remark}
Let us denote by $(2,1)\text{-Cat}$ the category of locally groupoidal 2-categories. It comes equipped with a canoncial model structure, too, as well as a canonical cartesian monoidal embedding
\[N\colon\mathrm{Ho}_{\infty}((2,1)\text{-Cat}))^{\times}\rightarrow\mathrm{Cat}_{\infty}^{\times}\]
of $\infty$-categories. This embedding automatically maps $\mathsf{E}_k$-monoids to $\mathsf{E}_k$-monoidal $\infty$-categories.
Furthermore, the functor $\mathrm{Alg}_{\mathsf{E}_1}(-)\colon\mathrm{Alg}_{\mathsf{E}_1}(\mathrm{Ho}_{\infty}(\twocat))^{\times}\rightarrow\mathrm{Ho}_{\infty}(\twocat))^{\times}$ restricts to a cartesian monoidal functor
\[\mathrm{Alg}_{\mathsf{E}_1}(-)\colon\mathrm{Alg}_{\mathsf{E}_1}(\mathrm{Ho}_{\infty}((2,1)\text{-Cat}))^{\times}\rightarrow\mathrm{Ho}_{\infty}((2,1)\text{-Cat}))^{\times}.\]
In analogy to Proposition~\ref{prop_alg=hom_1}, one can then show that for every $\mathsf{E}_m$-algebra $\mathcal{C}$ in
$\mathrm{Ho}_{\infty}((2,1)\text{-Cat}))^{\times}$, there is a natural equivalence
\[N(\mathrm{Alg}_{\mathsf{E}_m}(\mathcal{C}))\simeq\mathrm{Alg}_{\mathsf{E}_m}(N(\mathcal{C})).\]
One can make a similar observation when one more generally considers any given 2-category as an $(\infty,2)$-category, and any given Gray 
monoid as a monoidal $(\infty,2)$-category accordingly by way of Lemma~\ref{lemma_grmon=E1}. This justifies the notation
$\mathrm{Alg}_{\mathsf{E}_1}(-)$ for the functor $\mathrm{Ho}_{\infty}(\mathrm{PsMon}(-))$.
\end{remark}

\subsection*{Semi-strict algebras in the 3-category of 2-categories}

In this section we give a proof of Theorem~\ref{mainthmgray2}, which is expressed in algebraic terms yet to be specified precisely. We want 
to use Theorem~\ref{thm_main_12} to do so, which in turn is inherently phrased in homotopy theoretical terms. Thus, to deduce 
Theorem~\ref{mainthmgray2} from Theorem~\ref{thm_main_12}, we need a translation between the corresponding algebraic and homotopical 
notions. Given the nature of these homotopical notions, the algebraic structures involved need to be equivalence invariant. We therefore 
make the following notational convention.

\begin{notation}\label{notation_proofs1.2}
In the following, the term ``monoidal structure'' on a 2-category $\mathcal{C}$ refers to a bicategorical monoidal structure rather than a
Gray monoidal structure on $\mathcal{C}$. That is to say, it refers to a 2-category $\mathcal{C}$ that is 2-equivalent to a Gray monoid. 
The terminology is in line with the corresponding distinction between monoidal 2-functors and Gray functors of \cite{daystreet}. In this 
respect, a ``braided/sylleptic/symmetric monoidal 2-category'' will refer to a braided/sylleptic/symmetric monoidal bicategory whose 
underlying bicategory is a 2-category. All of these notions are hence to be understood as ``fully weak''.
We make the same notational convention about 2-functors. I.e.\ let $\mathcal{C}$ and $\mathcal{D}$ be 2-categories equipped with a
(braided/sylleptic/symmetric) monoidal structure each. Let $\bar{\mathcal{C}}$ be the (fully weakly braided/sylleptic/symmetric) Gray monoid 
together with a 2-equivalence $\bar{\mathcal{C}}\rightarrow\mathcal{C}$ that equips $\mathcal{C}$ with its corresponding monoidal structure. 
Let $\bar{\mathcal{D}}\rightarrow\mathcal{D}$ be accordingly. Then a (braided/sylleptic/symmetric) monoidal 2-functor
$f\colon\mathcal{C}\rightarrow\mathcal{D}$ is a 2-functor $f\colon\mathcal{C}\rightarrow\mathcal{D}$ together with a
(braided/sylleptic/symmetric) monoidal 2-functor $\bar{f}\colon\bar{\mathcal{C}}\rightarrow\bar{\mathcal{D}}$ in the sense of 
\cite{daystreet} that makes the associated square commute up to equivalence.
\end{notation}


In Lemma~\ref{lemma_grmon=E1} we have seen that for every 2-category $\mathcal{C}$ there is a 1-1 correspondence between monoidal 
structures on $\mathcal{C}$, and $\mathsf{E}_1$-monoid structures on $\mathcal{C}$ in $\mathrm{Ho}_{\infty}(\twocat)^{\times}$. Given two 
monoidal 2-categories $\mathcal{C}$ and $\mathcal{D}$ and a 2-functor $f\colon\mathcal{C}\rightarrow\mathcal{D}$, the same applies to 
monoidal structures on $f$ and $\mathsf{E}_1$-monoid structures on $f$ in $\mathrm{Ho}_{\infty}(\twocat)^{\times}$.
This 1-1 correspondence can be extended as follows \cite{rs_halgicatI}.

\begin{theorem}\label{thm_translatealghom}
There are the following bijections between sets of respective equivalence classes:
\begin{enumerate}
\item $\{$monoidal bicategories$\}\cong\{\mathsf{E}_1$-algebras in $\mathrm{BiCat}^{\times}\}$;
\item $\{$braided monoidal bicategories$\}\cong\{\mathsf{E}_2$-algebras in $\mathrm{BiCat}^{\times}\}$;
\item $\{$sylleptic monoidal bicategories$\}\cong\{\mathsf{E}_3$-algebras in $\mathrm{BiCat}^{\times}\}$;
\item $\{$symmetric monoidal bicategories$\}\cong\{\mathsf{E}_{\infty}$-algebras in $\mathrm{BiCat}^{\times}\}$.
\end{enumerate}
These bijections are natural with respect to the obvious forgetful maps from bottom to top. The same applies to the corresponding 
notions of morphisms between them.
\end{theorem}

Theorem~\ref{mainthmgray} can now be derived from Theorem~\ref{thm_main_12} as follows.

\begin{theorem}\label{mainthmgray}
Let $\mathcal{C}$ be a monoidal bicategory and $\mathrm{Mon}(\mathcal{C})$ be the bicategory of monoids in $\mathcal{C}$.
\begin{enumerate}
\item Every braiding on $\mathcal{C}$ induces a monoidal structure on $\mathrm{Mon}(\mathcal{C})$.
\item Every syllepsis on $\mathcal{C}$ induces a braided monoidal structure on $\mathrm{Mon}(\mathcal{C})$.
\item Every symmetry on $\mathcal{C}$ induces a symmetric monoidal structure on $\mathrm{Mon}(\mathcal{C})$.
\end{enumerate}
In each case the corresponding structure is canonically induced, i.e.\ the forgetful functor
$\mathrm{Mon}(\mathcal{C})\rightarrow\mathcal{C}$ is as monoidal as its domain is.
\end{theorem}

\begin{proof}
Assuming that the monoidal bicategory $\mathcal{C}$ is a Gray monoid by way of Theorem~\ref{thm_bito2cat}, we are to show that
\begin{enumerate}
\item Every braiding on $\mathcal{C}$ induces a monoidal structure on $\mathrm{PsMon}(\mathcal{C})$.
\item every syllepsis on $\mathcal{C}$ induces a braided monoidal structure on $\mathrm{PsMon}(\mathcal{C})$.
\item every symmetry on $\mathcal{C}$ induces a symmetric monoidal structure on $\mathrm{PsMon}(\mathcal{C})$.
\end{enumerate}
Here we employ the conventions from Remark~\ref{rem_defbicatmon} and Notation~\ref{notation_proofs1.2}. Furthermore, we are to show that
the forgetful functor
\[\mathrm{PsMon}(\mathcal{C})\xrightarrow{\pi_{\mathcal{C}}}\mathcal{C}\]
is (braided/symmetric) monoidal whenever $\mathcal{C}$ is braided (sylleptic/symmetric).

For Part (1), assume that $\mathcal{C}$ is a braided Gray monoid. By Part (1) of Theorem~\ref{thm_translatealghom} and 
Lemma~\ref{lemma_grmon=E1} it gives rise to an $\mathsf{E}_1$-monoid $\mathcal{C}$ in $\mathrm{Ho}_{\infty}(\mathrm{GrMon})^{\times}$. It 
follows that the 2-category $\mathrm{PsMon}(\mathcal{C})$ comes equipped with an $\mathsf{E}_1$-monoid structure in
$\mathrm{Ho}_{\infty}(\twocat)^{\times}$ by Corollary~\ref{corpsmonmonoidal}. Thus, there is a cofibrant Gray monoid
$\mathbb{L}\mathrm{PsMon}(\mathcal{C})$ together with an equivalence
$U(\mathbb{L}\mathrm{PsMon}(\mathcal{C}))\rightarrow\mathrm{PsMon}(\mathcal{C})$ of 2-categories again by
Lemma~\ref{lemma_grmon=E1} (in tandem with model categorical formalities). 
Lastly, the composite 2-functor
\begin{align}\label{equ1_mainthmgray}
\mathbb{L}\mathrm{PsMon}(\mathcal{C})\xrightarrow{\simeq}\mathrm{PsMon}(\mathcal{C})\xrightarrow{\pi_{\mathcal{C}}}\mathcal{C}
\end{align}
lifts to a map of $\mathsf{E}_1$-monoids in $\mathrm{Ho}_{\infty}(\twocat)^{\times}$ by Corollary~\ref{cor_monforget}. It follows that the 
2-functor (\ref{equ1_mainthmgray}) is equivalent to a Gray functor by Lemma~\ref{lemma_grmon=E1}, and hence is monoidal. This proves Part 
(1). 

Parts (2) and (3) follow the same pattern.
For Part (2), assume $\mathcal{C}$ is a sylleptic Gray monoid. By Part (2) of Theorem~\ref{thm_translatealghom} it gives 
rise to an $\mathsf{E}_2$-monoid $\mathcal{C}$ in $\mathrm{Ho}_{\infty}(\mathrm{GrMon})^{\times}$. It follows that the 2-category
$\mathrm{PsMon}(\mathcal{C})$ comes equipped with an $\mathsf{E}_2$-monoid structure in $\mathrm{Ho}_{\infty}(\twocat)^{\times}$ by 
Corollary~\ref{corpsmonmonoidal}. In particular, there is a cofibrant Gray monoid $\mathbb{L}\mathrm{PsMon}(\mathcal{C})$ together with an 
equivalence $U(\mathbb{L}\mathrm{PsMon}(\mathcal{C}))\rightarrow\mathrm{PsMon}(\mathcal{C})$ of 2-categories by 
Lemma~\ref{lemma_grmon=E1}. The Gray monoid $\mathbb{L}\mathrm{PsMon}(\mathcal{C})$ is braided by the converse of Part (2) of
Theorem~\ref{thm_translatealghom}. Lastly, the composite 2-functor
\begin{align}\label{equ2_mainthmgray}
\mathbb{L}\mathrm{PsMon}(\mathcal{C})\xrightarrow{\simeq}\mathrm{PsMon}(\mathcal{C})\xrightarrow{\pi(\mathcal{C})}\mathcal{C}
\end{align}
lifts to a map of $\mathsf{E}_2$-monoids in $\mathrm{Ho}_{\infty}(\twocat)^{\times}$ by Corollary~\ref{cor_monforget}. It follows that the 
2-functor (\ref{equ2_mainthmgray}) is braided monoidal by Theorem~\ref{thm_translatealghom}. This finishes Part (2).

For Part (3), assume $\mathcal{C}$ is a symmetric Gray monoid. By Part (3) of Theorem~\ref{thm_translatealghom} it gives 
rise to an $\mathsf{E}_{\infty}$-monoid $\mathcal{C}$ in $\mathrm{Ho}_{\infty}(\twocat)^{\times}$. As
$\mathrm{Ho}_{\infty}(\twocat)$ is a $3$-category, the natural forgetful functor $\mathrm{Alg}_{\mathsf{E}_{\infty}}(\mathrm{Ho}_{\infty}(\twocat)^{\times})\rightarrow\mathrm{Alg}_{\mathsf{E}_k}(\mathrm{Ho}_{\infty}(\twocat)^{\times})$ is an equivalence for all $k\geq 4$ \cite[Theorem 5.1.2.2, Example 5.1.2.3]{lurieha}. 
The same applies to $\mathrm{Alg}_{\mathsf{E}_{\infty}}(\mathrm{Ho}_{\infty}(\mathrm{GrMon})^{\times})$.
It follows that $\mathcal{C}$ is equivalently an $\mathsf{E}_{\infty}$-monoid $\mathcal{C}$ in
$\mathrm{Ho}_{\infty}(\mathrm{GrMon})^{\times}$, and that the 2-category $\mathrm{PsMon}(\mathcal{C})$ comes equipped with an
$\mathsf{E}_{\infty}$-monoid structure in
$\mathrm{Ho}_{\infty}(\twocat)$ as well by way of Corollary~\ref{corpsmonmonoidal}. Thus, there is a cofibrant Gray monoid
$\mathbb{L}\mathrm{PsMon}(\mathcal{C})$ together with an equivalence
$U(\mathbb{L}\mathrm{PsMon}(\mathcal{C}))\rightarrow\mathrm{PsMon}(\mathcal{C})$ of 2-categories by Lemma~\ref{lemma_grmon=E1}. The Gray 
monoid $\mathbb{L}\mathrm{PsMon}(\mathcal{C})$ is symmetric by the converse of Theorem~\ref{thm_translatealghom}.3. Lastly, the composite 
2-functor
\begin{align}\label{equ3_mainthmgray}
\mathbb{L}\mathrm{PsMon}(\mathcal{C})\xrightarrow{\simeq}\mathrm{PsMon}(\mathcal{C})\xrightarrow{\pi(\mathcal{C})}\mathcal{C}
\end{align}
lifts to a map of $\mathsf{E}_{\infty}$-monoids in $\mathrm{Ho}_{\infty}(\twocat){^\times}$ by Corollary~\ref{cor_monforget}. 
It follows that the 2-functor (\ref{equ3_mainthmgray}) is sylleptic (and hence symmetric) monoidal by Theorem~\ref{thm_translatealghom}. 
This finishes Part (3).
\end{proof}

\begin{remark}
Proposition~\ref{prop_mult_explicit} provides a blueprint to extract a concrete formula for the multiplication of two pseudomonoids in a 
braided monoidal bicategory from Theorem~\ref{mainthmgray} as well.
\end{remark} 

\section{Future work}\label{sec_future}

Given a Gray monoid $\mathcal{C}$, we recall the definition of the 2-category $\mathrm{BrMon}(\mathcal{C})$ of braided
pseudomonoids in $\mathcal{C}$ whenever $\mathcal{C}$ itself is braided, as well as the 2-category $\mathrm{CMon}(\mathcal{C})$ of 
symmetric pseudomonoids in $\mathcal{C}$ whenever $\mathcal{C}$ itself is symmetric from \cite{mccrudden}. The following 
characterisation thereof may be thought of as an algebraic 2-dimensional version of the Additivity Theorem, which in this paper we will 
however not pursue further.

\begin{conjecture}\label{conj_higherEnint}
Let $\mathcal{C}$ be a Gray monoid. 
\begin{enumerate}
\item Every braiding on $\mathcal{C}$ induces an equivalence $\mathrm{Alg}_{\mathsf{E}_2}(\mathcal{C})\simeq\mathrm{BrMon}(\mathcal{C})$
of 2-categories.
\item Every syllepsis on $\mathcal{C}$ induces an equivalence $\mathrm{Alg}_{\mathsf{E}_3}(\mathcal{C})\simeq\mathrm{CMon}(\mathcal{C})$
of 2-categories.
\end{enumerate}
\end{conjecture}

\begin{remark}
Conjecture~\ref{conj_higherEnint} is a direct generalisation of \cite[Remark 5.1]{joyalstreet} from the Gray monoid
$\mathrm{Cat}$ of categories equipped with its cartesian monoidal structure to a general Gray monoid $\mathcal{C}$. That is to say, loc.\ 
cit.\ is precisely Conjecture~\ref{conj_higherEnint} for the Gray monoid $\mathcal{C}=\mathrm{Cat}^{\times}$.
Similarly, Theorem~\ref{thm_translatealghom} itself is also a generalization of \cite[Remark 5.1]{joyalstreet}, however in an orthogonal 
sense, to the tricategory of bicategories.
\end{remark}

Given that Conjecture~\ref{conj_higherEnint} holds one can show the following tiered version of Theorem~\ref{mainthmgray}.

\begin{theorem}\label{mainthmgray2}
Let $\mathcal{C}$ be a monoidal bicategory and $\mathrm{Mon}(\mathcal{C})$ be the bicategory of monoids in $\mathcal{C}$.
Assume Conjecture~\ref{conj_higherEnint} holds.
\begin{enumerate}
\item Every syllepsis on $\mathcal{C}$ induces a monoidal structure on the bicategory $\mathrm{BrMon}(\mathcal{C})$ of braided monoids in
$\mathcal{C}$.
\item Every symmetry on $\mathcal{C}$ induces
\begin{enumerate}
\item a symmetric monoidal structure on $\mathrm{BrMon}(\mathcal{C})$, and 
\item a symmetric monoidal structure on the bicategory $\mathrm{CMon}(\mathcal{C})$ of symmetric monoids. 
\end{enumerate}
\end{enumerate}
In each case the corresponding structure is canonically induced, i.e.\ the forgetful functors
$\mathrm{BrMon}(\mathcal{C})\rightarrow\mathcal{C}$ and $\mathrm{CMon}(\mathcal{C})\rightarrow\mathcal{C}$ are as monoidal as their domain 
is.
\end{theorem}
\begin{proof}
Again assuming that the monoidal bicategory $\mathcal{C}$ is a Gray monoid by way of Theorem~\ref{thm_bito2cat}, we are to show that
\begin{enumerate}
\item Every syllepsis on $\mathcal{C}$ induces a monoidal structure on the 2-category $\mathrm{BrMon}(\mathcal{C})$.
\item Every symmetry on $\mathcal{C}$ induces 
\begin{enumerate}
\item a symmetric monoidal structure on $\mathrm{BrMon}(\mathcal{C})$, and 
\item a symmetric monoidal structure on the bicategory $\mathrm{CMon}(\mathcal{C})$. 
\end{enumerate}
\end{enumerate}
Furthermore, we are to show that the forgetful functor $\mathrm{BrMon}(\mathcal{C})\rightarrow\mathcal{C}$ is braided (symmetric) monoidal 
whenever $\mathcal{C}$ is sylleptic (symmetric). The forgetful functor $\mathrm{CMon}(\mathcal{C})\rightarrow\mathcal{C}$ is symmetric 
monoidal whenever $\mathcal{C}$ is symmetric.

By way of Conjecture~\ref{conj_higherEnint}, to do so we can simply apply Theorem~\ref{mainthmgray} iteratively. For instance, let
$\mathcal{C}$ be a Gray monoid with a syllepsis. We obtain a sequence of 2-functors as follows.
\begin{align}\label{equ_mainthmgray2}
\mathrm{Alg}_{\mathsf{E}_2}(\mathcal{C})\xrightarrow{\simeq}\mathrm{PsMon}(\mathrm{Alg}_{\mathsf{E}_1}(\mathcal{C}))\xrightarrow{\pi_{\mathrm{Alg}_{E_1}(\mathcal{C})}}\mathrm{Alg}_{E_1}(\mathcal{C})\xrightarrow{\simeq}\mathrm{PsMon}(\mathcal{C})\xrightarrow{\pi_{\mathcal{C}}}\mathcal{C}
\end{align}
Here, both $\mathrm{PsMon}(\mathcal{C})$ and $\pi_{\mathcal{C}}\colon\mathrm{PsMon}(\mathcal{C})\rightarrow\mathcal{C}$ are braided 
monoidal by Theorem~\ref{mainthmgray}. Furthermore, $\mathrm{Alg}_{\mathsf{E}_1}(\mathcal{C})\xrightarrow{\simeq}\mathrm{PsMon}(\mathcal{C})$ is the 
braided Gray monoid that equips the 2-category $\mathrm{PsMon}(\mathcal{C})$ with its braided monoidal structure. Thus, again by 
Theorem~\ref{mainthmgray}, both $\mathrm{PsMon}(\mathrm{Alg}_{\mathsf{E}_1}(\mathcal{C}))$ and
\[\pi_{\mathrm{Alg}_{\mathsf{E}_1}(\mathcal{C})}\colon\mathrm{PsMon}(\mathrm{Alg}_{\mathsf{E}_1}(\mathcal{C}))\rightarrow\mathrm{Alg}_{\mathsf{E}_1}(\mathcal{C})\]
are monoidal. And again, $\mathrm{Alg}_{\mathsf{E}_2}(\mathcal{C})\xrightarrow{\simeq}\mathrm{PsMon}(\mathrm{Alg}_{\mathsf{E}_1}(\mathcal{C}))$ is the Gray 
monoid that equips the 2-category $\mathrm{PsMon}(\mathrm{Alg}_{\mathsf{E}_1}(\mathcal{C}))$ with its monoidal structure. It follows that the entire 
composition (\ref{equ_mainthmgray2}) is monoidal. The same argumentation applies to the two other cases.
\end{proof}

\bibliographystyle{amsalpha}
\bibliography{BSBib}

\newcommand{\noopsort}[1]{}
\providecommand{\bysame}{\leavevmode\hbox to3em{\hrulefill}\thinspace}
\providecommand{\MR}{\relax\ifhmode\unskip\space\fi MR }
\providecommand{\MRhref}[2]{%
  \href{http://www.ams.org/mathscinet-getitem?mr=#1}{#2}
}
\providecommand{\href}[2]{#2}
\begin{thebibliography}{KV94b}

\bibitem[BD98a]{baezdolan_cat}
J.~Baez and J.~Dolan, \emph{Categorification}, Contemp. Math. \textbf{230}
  (1998), 1–36.

\bibitem[BD98b]{baezdolan3}
\bysame, \emph{Higher-dimensional algebra {III}: n-{C}ategories and the algebra
  of opetopes}, Advances in Mathematics \textbf{135} (1998), no.~2, 145--206.

\bibitem[DS97]{daystreet}
B.~Day and R.~Street, \emph{Monoidal bicategories and {H}opf algebroids},
  Advances in Mathematics \textbf{129} (1997), 99--157.

\bibitem[DY23]{decoppetyu2cat}
T.D. Décoppet and M.~Yu, \emph{Gauging noninvertible defects: a 2-categorical
  perspective}, Letters in Mathematical Physics \textbf{113} (2023), no.~36.

\bibitem[GL21]{gambinolobbia}
N.~Gambino and G.~Lobbia, \emph{On the formal theory of pseudomonads and
  pseudodistributive laws}, Theory and {A}pplications of {C}ategories
  \textbf{37} (2021), no.~2, 14--56.

\bibitem[GPS95]{gps_coherence}
R.~Gordon, A.~J. Power, and R.~Street, \emph{Coherence for tricategories},
  Memoirs of the AMS \textbf{117} (1995), no.~558, 1--85.

\bibitem[Gur07]{gurski_stability}
N.~Gurski, \emph{Stability for pseudomonoids},
  https://www.mat.uc.pt/~categ/ct2007/abstracts/Gurski.pdf, 2007, Abstract for
  Category Theory 2007 in Algarve, Portugal.

\bibitem[Gur13]{gurski_3coh}
N.~Gurksi, \emph{Coherence in three-dimensional category theory}, Cambridge
  University Press, 2013.

\bibitem[Hof11]{hoffnung_tetra}
A.~E. Hoffnung, \emph{Spans in 2-categories: {A} monoidal tricategory},
  arXiv:1112.0560, 2011, v2 uploaded 18 Sep 2013.

\bibitem[Hov99]{hovey}
M.~Hovey, \emph{Model categories}, Mathematical Surveys and Monographs,
  vol.~63, American Mathematical Society, 1999.

\bibitem[JS93]{joyalstreet}
A.~Joyal and R.~Street, \emph{Braided tensor categories}, Advances in
  Mathematics \textbf{102} (1993), 20--78.

\bibitem[JY21]{johnsonyau}
N.~Johnson and D.~Yau, \emph{2-{D}imensional {C}ategories}, Oxford University
  Press, 2021.

\bibitem[KV94a]{kapranovvoevodskyI}
M.~Kapranov and V.~Voevodsky, \emph{2-{C}ategories and {Z}amolodchikov
  tetrahedra}, Proc. Symp. Pure Math. \textbf{56} (1994), 177--260.

\bibitem[KV94b]{kapranovvoevodskyII}
\bysame, \emph{Braided monoidal 2-categories and {M}anin--{S}chechtman higher
  braid groups}, J. Pure Appl. Algebra \textbf{92} (1994), 241--267.

\bibitem[Lac00]{lackpsm}
S.~Lack, \emph{A coherent approach to pseudomonads}, Advances in Mathematics
  \textbf{152} (2000), 179–202.

\bibitem[Lac02]{lack2catms}
\bysame, \emph{A {Q}uillen model structure for 2-categories}, K-Theory
  \textbf{26} (2002), no.~2, 171–205.

\bibitem[Lac11]{lackgraycatms}
\bysame, \emph{A {Q}uillen model structure for {G}ray-categories}, K-Theory
  \textbf{8} (2011), no.~2, 183–221.

\bibitem[Lan71]{maclane}
S.~Mac Lane, \emph{Categories for the working mathematician}, vol.~5, Springer,
  New York, 1971.

\bibitem[Lur17]{lurieha}
J.~Lurie, \emph{Higher algebra},
  \url{http://www.math.harvard.edu/~lurie/papers/HA.pdf}, 2017, Last update 18
  September 2017.

\bibitem[May72]{may_loop}
J.~P. May, \emph{The geometry of iterated loop spaces}, Lecture Notes in
  Mathematics, no. 271, Springer-Verlag, Berlin, New York, 1972.

\bibitem[McC00]{mccrudden}
J.C. McCrudden, \emph{Balanced coalgebroids}, Theory and Applications of
  Categories \textbf{7} (2000), no.~6, 71--147.

\bibitem[SP09]{schommerpries_thesis}
C.~J. Schommer-Pries, \emph{The classification of two-dimensional extended
  topological field theories}, Ph.D. thesis, University of California,
  Berkeley, Berkeley, CA 94720-3840, United States, 2009.

\bibitem[SS00]{schwedeshipley_mon}
S.~Schwede and B.~Shipley, \emph{Algebras and modules in monoidal model
  categories}, Proc. London Math. Soc. \textbf{80} (2000), no.~3, 491--511.

\bibitem[Ste26]{rs_halgicatI}
R.~Stenzel, \emph{The higher algebra and geometry of bicategories},
  arXiv:2602.14424, 2026.

\end{thebibliography}
\end{document}